\documentclass[10pt]{amsart}
\usepackage[utf8]{inputenc}
\usepackage{amssymb}
\usepackage{graphicx}
\usepackage{mathtools}
\usepackage{enumitem}
\usepackage[foot]{amsaddr}
\usepackage{cite}

\newtheorem{theorem}{Theorem}
\newtheorem{proposition}[theorem]{Proposition}
\newtheorem{lemma}[theorem]{Lemma}
\newtheorem{corollary}[theorem]{Corollary}

\title[Agro-ecological control: optimizing the harvest]{Agro-ecological control of a pest-host system: optimizing the harvest}

\author{Baptiste Maucourt}
\address[B. M.]{Institut Camille Jordan, Universit\'{e} Claude Bernard Lyon-1, 43 boulevard du 11 novembre 1918, 69622 Villeurbanne Cedex, France}
\email{maucourt@math.univ-lyon1.fr}

\begin{document}


\begin{abstract}
We delve into the interactions between a prey-predator and a vector-borne epidemic system, driven by agro-ecological motivations. This system involves an ODE, two reaction--diffusion PDEs and one reaction--diffusion--advection PDE. It has no complete variational or monotonic structure and features spatially heterogeneous coefficients. Our initial focus is to examine the continuity of a quantity known as "harvest", which depends on the time-integral of infected vectors. We analyze its asymptotic behaviour as the domain becomes homogeneous. Then we tackle a non-standard optimal control problem related to the linearized harvest and conduct an analysis to establish the existence, uniqueness, and properties of optimizers. Finally, we refine the location of optimizers under specific initial conditions.
\end{abstract}

\maketitle

\section{Introduction}

We are interested in the sugar beet agro-ecosystem, particularly how aphids spread four yellow viruses (BMYV, BChV, BYV, BtMV) in beet fields. Controlling these viruses is crucial, as they can reduce sugar concentration in beets by fifty percent. Pesticides are commonly used but harm biodiversity and health. An alternative approach involves using natural aphid predators like ladybugs (Hippodamia variegata or Chnootriba similis) in flower patches among beet fields. These "biodiversity refuges" can boost predator populations and control aphids, but reduces the total number of beets present in the field. We seek optimal refuge geometry to maximize the harvest of non-infected beets, treating it as an optimal control problem. Many similar agro-ecosystems can benefit from our analysis.

This work is a follow-up analysis from \cite{girardinmaucourt2023}, in which one can find a more thorough biological context of our problem in the introduction, as well as comments on the ideal free dispersal strategy of the predators studied in \cite{cantrellcosnerlou2010}. In our initial paper, we primarily presented numerical arguments regarding the features of harvest optimizers. While a uniform refuge optimizes a linear criterion, we illustrated that its optimality seems not always guaranteed, depending on the initial distribution of infected vectors. For example, when infected vectors are predominantly concentrated in the center of the field, a centrally positioned refuge results in a superior harvest compared to a uniform one. Here, we provide an analytical proof of this phenomenon (Theorem \ref{threarrange}), considering a symmetric, radially decreasing initial condition known as Schwarz rearrangement. However, when infected vectors are, on average, initially homogeneously positioned, a constant refuge is optimal (Theorem \ref{thconstant}). Since having a constant refuge seems biologically unrealistic, we provide a way to approach it through homogenization (Theorem \ref{thhomogen}).

Let $\Omega$ be an open bounded connected set with Lipschitz boundary $\partial\Omega$ in a Euclidean space $\mathbb{R}^d$. Below is the parabolic--ordinary differential system of interest.
{\small \begin{equation}\label{eqspring}
 \left\lbrace \begin{array}{lll}
 \partial_tI & = & \beta_{VH}(H(x)-I)V_i \\
 \partial_tV_i & = & \sigma_V\Delta V_i+\beta_{HV}IV_s-\alpha V_i-d_V(x)V_i-s_V(V_s+V_i)V_i-hPV_i \\
 \partial_tV_s & = & \sigma_V\Delta V_s-\beta_{HV}IV_s+\alpha V_i-d_V(x)V_s-s_V(V_s+V_i)V_s-hPV_s\\
 && +b_V(x)(V_s+V_i) \\
 \partial_t P & = & \sigma_P\nabla\cdot\left(r_P(x)\nabla\left(\dfrac{P}{r_P(x)}\right)\right)+\gamma h(V_s+V_i)P+r_P(x)P-s_PP^2
 \end{array}\right.
\end{equation}}
in $\Omega\times\mathbb{R}_+^*$, where we omitted the dependency on $(x,t)$ for $I$, $V_i$, $V_s$ and $P$.

Here, the notations $I$, $V_i$, $V_s$, $P$ stand respectively for the population density of infected hosts, infected vectors, susceptible vectors and predators. The various parameters are all biologically meaningful: $H\in\mathcal{C}^{2+\alpha}(\overline{\Omega},\mathbb{R}_+^*)$ is the total population of hosts, $\beta_{VH},\beta_{HV}>0$ are the transmission rates (from vectors to hosts and vice-versa), $\sigma_V,\sigma_P>0$ are the diffusion rates of the vectors and predators, respectively, $\alpha$ is the recovery rate of the vectors, $b_V,d_V\in\mathcal{C}^{2+\alpha}(\overline{\Omega},\mathbb{R}_+^*)$ are the birth rate and death rate of the vectors, respectively, $r_P\in\mathcal{C}^{2+\alpha}(\overline{\Omega},\mathbb{R}_+^*)$ is the Malthusian growth rate of the predators, $s_V,s_P>0$ are the saturation rates of vectors and predators, respectively (intraspecific competition for space or resources leading to logistic growth), $h>0$ is the predation rate and $\gamma>0$ is a coefficient measuring the efficiency of predation.

We supplement it with initial conditions:
\[
(I, V_i, V_s, P)(\cdot,0) = (0, V_{i,0}, V_{s,0}, P_0)\quad\text{in }\Omega,
\] 
where $V_{i,0}$, $V_{s,0}$, $P_0$ are nonnegative functions in $L^{\infty}(\overline{\Omega})$, and with ``no-flux'' boundary conditions:
\[
\dfrac{\partial V_i}{\partial n}=\dfrac{\partial V_s}{\partial n}=\dfrac{\partial (P/r_P)}{\partial n}=0\quad\text{on }\partial\Omega.
\]

Note that the condition $\partial_n(P/r_P)=0$ reduces to $\partial_n P=0$ under the biologically relevant assumption $\partial_n r_P=0$, that will be a standing assumption from now on. Hence, we actually have homogeneous Neumann boundary conditions for all three animal populations $V_i$, $V_s$ and $P$.

When $\Omega$ has a smooth boundary, by \cite{girardinmaucourt2023}, the system \eqref{eqspring} admits a unique classical solution defined in $\overline{\Omega}\times\mathbb{R}_+^*$. In our case, the well-posedness of this system is only a conjecture.

Let $R\in\mathcal{C}^{2+\alpha}(\overline{\Omega},[0,1])$. We will call the function $R$ the refuge. In the rest of the paper, we assume that the growth rate of the predators and the preys, and the population of hosts are of the form
\[
\begin{array}{l}
 r_P=(r_P^r-r_P^f)R+r_P^f\\
 b_V=(b_V^r-b_V^f)R+b_V^f\\
 d_V=(d_V^r-d_V^f)R+d_V^f\\
 r_V=b_V-d_V=(r_V^r-r_V^f)R+r_V^f\\
 H=H^0(1-R),
\end{array}
\]
with $r_P^r,r_P^f,b_V^r,b_V^f,d_V^r,d_V^f,r_V^r,r_V^f,H^0>0$.

With this set of notations and assumptions, our optimal control problem can be recasted as follows:

\textbf{Optimal control problem:}
Characterize the set of optimal $R$ such that the harvest of healthy beets $\eta=\lim_{t\to+\infty}\int_{\Omega} H-I(\cdot,t)=\int_{\Omega}He^{-\beta_{VH}\int_0^{+\infty}V_i(\cdot,t)\mathrm{d}t}$ is maximized.

It is important to highlight that in Section 4, dedicated to optimizing $\eta$, Theorem \ref{thschwarz} and \ref{corschwarz} require the shape of the refuge to reside within a subset of $\mathrm{L}^{\infty}(\Omega)$ for the sake of compactness. We acknowledge that this requirement somewhat sets this result apart from the rest of the paper. Nevertheless, given that the well-posedness of the system should not be an issue when the coefficients are in $\mathrm{L}^{\infty}(\Omega)$, we believe that this result remains appropriate.

\subsection{Organization of the paper}

In the following subsections, we introduce our notations and present our main results.
In Section 2, we investigate the continuity of the harvest and provide partial answers to the question.
Section 3 is dedicated to proving our primary homogenization result.
Finally, in Section 4, we establish our key findings regarding the location of optimizers based on the properties of $V_{i,0}$.

\subsection{Notations}

\begin{itemize}
 \item $V_0=V_{i,0}+V_{s,0}$ and $V=V_i+V_s$.
 \item $r_V=b_V-d_V$.
 \item $\overline{\mathcal{L}}=\nabla\cdot\left(r_P\nabla\left(\dfrac{\cdot}{r_P}\right)\right)$.
 \item $\mathcal{L}_{V_s}=-\sigma_V\Delta-r_V+h\dfrac{r_P}{s_P}$.
 \item $\lambda_1(\mathcal{L}_{V_s})$ the principal eigenvalue of the operator $\mathcal{L}_{V_s}$ with Neumann boundary conditions, associated with the positive principal eigenfunction $\varphi_{s,1}$ (with $\min_{\Omega}\varphi_{s,1}=1$):
 \[
  \left\lbrace \begin{array}{ll}
   \mathcal{L}_{V_s}(\varphi_{s,1})=-\sigma_V\Delta\varphi_{s,1}+(-r_V+h\dfrac{r_P}{s_P})\varphi_{s,1}=\lambda_1(\mathcal{L}_{V_s})\varphi_{s,1} & \mbox{ in }\Omega\\
   \dfrac{\partial\varphi_{s,1}}{\partial n}=0 & \mbox{ on }\partial\Omega
  \end{array}\right.
 \]
 \item $I_{\infty}:x\mapsto\lim_{t\to+\infty}I(x,t)$.
 \item The harvest $\eta=\int_{\Omega}H-I_{\infty}=\int_{\Omega}He^{-\beta_{VH}\int_0^{+\infty}V_i(\cdot,t)\mathrm{d}t}$.
 \item The linearized harvest $\eta_L$ (defined in Section 4).
 \item If $f$ is a function defined on $\Omega$, $\overline{f}=\max_{\overline{\Omega}}f$ and $\underline{f}=\min_{\overline{\Omega}}f$.
\end{itemize}

\subsection{Main result}

The idea of our main theorems is the following. The quantities $m$ and $\xi$ are defined in the beginning of section 4.1.

\begin{theorem}\label{thhomogen}
 When the frequency of the refuge goes to infinity, the harvest converges to the harvest of the homogenized system.
\end{theorem}

\begin{theorem}\label{threarrange}
 If $V_{i,0}$ is symmetric and radially decreasing, the optimal refuge of the linearized harvest is not constant in space, and there exists a symmetric, radially decreasing refuge better than the constant one.
\end{theorem}

\begin{theorem}\label{thconstant}
 If (the mean of) $V_{i,0}$ is constant, the optimal refuge of (the mean of) the linearized harvest is constant and equal to $\left(\dfrac{\sqrt{\beta_{VH}(\xi+m)\mathbb{E}[V_{i,0}]}-\xi}{m}\right)^+$.
\end{theorem}

The formal context and notations will be given in Section 3 and 4.

Our homogenization analysis uses an ansatz technique developed by Murat and Tartar \cite{murattartar1977}, which effectively handles the non-trivial aspects of the ideal free dispersal term in the predator equation. The remaining portion of the proof relies on favorable estimates of the solution and the application of the Aubin-Lions theorem on $\Omega\times[0,T]$, and linear estimates previously conducted in \cite{girardinmaucourt2023}, which ensures the convergence of the complete time-integral of infected vectors.

The work performed on a Schwarz rearranged initial condition involves the utilization of Talenti inequalities \cite{talenti2016}, a case of strict inequality in the extended Hardy–Littlewood inequalities \cite{hajaiej2004}, and the Pólya–Szegő inequality \cite{talenti1994}.

\section{Harvest around $\lambda_1(\mathcal{L}_{V_s})=0$}

This section deals with the continuity of the harvest $\eta$ with respect to $\lambda_1(\mathcal{L}_{V_s})$. In \cite{girardinmaucourt2023}, it is claimed that $\eta$ is null when $\lambda_1(\mathcal{L}_{V_s})$ is negative, and positive when $\lambda_1(\mathcal{L}_{V_s})$ is positive. It is trivial that $\lim_{\lambda_1(\mathcal{L}_{V_s})\to0^-}\eta=0$. In this section, we establish that the harvest is null when $\lambda_1(\mathcal{L}_{V_s})=0$ in a simpler scenario, specifically when $P=r_P/s_P$. If this seems to be a costly assumption regarding the behaviour of the predators, it is, in fact, not: we can arrive at this new system by setting $\gamma=0$ and $P_0=r_P/s_P$ in system \ref{eqspring}. We leave open the question regarding the behaviour of $\eta$ when $\lambda_1(\mathcal{L}_{V_s})$ approaches $0$ from the right. We conjecture that the harvest should be continuous with respect to $\lambda_1(\mathcal{L}_{V_s})$, therefore the right limit should be $0$. To support this conjecture, we examine a simpler case when $P=r_P/s_P$ and $-r_V+h\dfrac{r_P}{s_P}$ uniformly converges to $0$. For the complete prey-predator system, the maximum principle does not seem to lead anywhere, since the predation term in the prey equation lacks the correct sign, and the best predator supersolution we could obtain is of the form
\[
 \overline{P}(t)=\dfrac{\exp\left(t-\dfrac{e^{-\lambda_1(\mathcal{L}_{V_s})t}}{\lambda_1(\mathcal{L}_{V_s})}\right)}{1+\int_0^t\exp\left(s-\dfrac{e^{-\lambda_1(\mathcal{L}_{V_s})s}}{\lambda_1(\mathcal{L}_{V_s})}\right)\mathrm{d}s},
\]
which, while it can be integrated in time, appears to be inadequate in providing insights to the behaviour of $\int_0^{+\infty}V_i$ when $\lambda_1(\mathcal{L}_{V_s})$ approaches $0$. Moreover, when the refuge is not constant, our control over $\lambda_1(\mathcal{L}_{V_s})$ is not explicit, necessitating linear analysis to link this principal eigenvalue to $V_i$. This, again, does not seem to yield any conclusive results regarding the limit of $\int_0^{+\infty}V_i$ around $\lambda_1(\mathcal{L}_{V_s})=0$. Even numerically, as 1) the harvest $\eta$ can only be approximated by computing $\int_0^TV_i$ for a large $T$ and 2) $V_i$ seems to converge to $0$ at a rate similar to $1/t$ when $\lambda_1(\mathcal{L}_{V_s})$ is close to $0$, the continuity of $\eta$ is not clear at all. Some remarks on the numerical behaviour of $\eta$ around $\lambda_1(\mathcal{L}_{V_s})=0$ can be found in Section 4 (see Figure \ref{harvestfunofmu}).

\begin{proposition}
 Consider the system
 
{\small\begin{equation}\label{eqconstpred}
 \begin{cases}
 \partial_tI & =\beta_{VH}(H-I)V_i \\
 \partial_tV_i & =\sigma_V\Delta V_i+\beta_{HV}IV_s-\alpha V_i-d_VV_i-s_V(V_s+V_i)V_i-h\dfrac{r_P}{s_P}V_i \\
 \partial_tV_s & =\sigma_V\Delta V_s-\beta_{HV}IV_s+\alpha V_i-d_VV_s-s_V(V_s+V_i)V_s-h\dfrac{r_P}{s_P}V_s+b_V(V_s+V_i)
 \end{cases}
\end{equation}}

with nonnegative initial condition in $\mathrm{L}^{\infty}(\overline{\Omega})$ and Neumann boundary conditions (i.e system \eqref{eqspring} with the predators replaced by $r_P/s_P$), with its associated harvest $\eta$. Suppose that $V_{i,0}\not\equiv 0$ and $\lambda_1(\mathcal{L}_{V_s})=0$.

Then, $\eta=0$.
\end{proposition}

\begin{proof}
 Recall that
 \[
  \eta=\int_{\Omega}He^{-\beta_{VH}\int_0^{+\infty}V_i(\cdot,t)\mathrm{d}t}.
 \]
 
 It is sufficient to prove that the integral $\int_0^{+\infty}V_i(x,t)$ does not converge for any $x\in\Omega$. Let us first deal with the convergence of $V_i$ to $0$ when $t$ goes to infinity.
 
 By summing the last two equations of \eqref{eqconstpred}, $V$ solves the equation
 \begin{equation}\label{preyeq}
  \partial_tV-\sigma_V\Delta V =(r_V-h\dfrac{r_P}{s_P})V-s_VV^2
 \end{equation}
 for $(x,t)\in\Omega\times\mathbb{R}_+^*$, with Neumann boundary conditions and initial condition $V(\cdot,0)=V_0$.
 
 Let $\varphi_{s,1}$ be the positive principal eigenvalue associated with $\lambda_1(\mathcal{L}_{V_s})$, and let $\overline{V}=\dfrac{2\varphi_{s,1}}{s_V(t+1)}$. Because $\lambda_1(\mathcal{L}_{V_s})$ is null, $-\Delta\varphi_{s,1}=(r_V-h\dfrac{r_P}{s_P})\varphi_{s,1}$, hence $\overline{V}$ solves
 \[
  \partial_t\overline{V}-\sigma_V\Delta\overline{V}=\left(r_V-h\dfrac{r_P}{s_P}\right)\overline{V}-\dfrac{1}{t+1}\overline{V}\geq\left(r_V-h\dfrac{r_P}{s_P}\right)\overline{V}-s_V\overline{V}^2
 \]
 and is then a supersolution of the Cauchy--Neumann system \eqref{preyeq} solved by $V$ (recall that $\min_{\Omega}\varphi_{s,1}=1$). We conclude that, for all $x\in\Omega$ and $t>0$,
 \[
  V_i(x,t)\leq V(x,t)\leq\overline{V}(x,t)\leq\dfrac{2\overline{\varphi_{s,1}}}{s_V(t+1)}\underset{t\to+\infty}{\longrightarrow}0,
 \]
 hence the uniform convergence of $V_i$ to $0$.
 
 $V_i$ solves, in $\Omega\times\mathbb{R}_+^*$,
\[
 \partial_tV_i-\sigma_V\Delta V_i=\beta_{HV}I(V-V_i)-\alpha V_i-d_VV_i-s_VVV_i-h\dfrac{r_P}{s_P}V_i,
\]
 hence
 \[
  \partial_tV_i-\sigma_V\Delta V_i+(\beta_{HV}I+\alpha+d_V+s_VV+h\dfrac{r_P}{s_P})V_i=\beta_{HV}IV.
 \]
 Denote $\overline{C}=\beta_{HV}\overline{H}+\alpha+\overline{d_V}+s_V\overline{V}+h\dfrac{\overline{r_P}}{s_P}>0$ (we used the global boundedness of $V$ in $\Omega\times\mathbb{R}_+^*$ proved in \cite{girardinmaucourt2023}).
 Because $V_{i,0}\not\equiv 0$, by \cite{girardinmaucourt2023}, there exists $C>0$ such that $I(x,t)\geq C$ for all $(x,t)\in\Omega\times[1,+\infty)$. Therefore, in $\Omega\times[1,+\infty)$,
 \[
  \partial_tV_i-\sigma_V\Delta V_i+\overline{C}V_i\geq\beta_{HV}CV.
 \]
 Let $T>1$. Integrating the previous equation on $\Omega$ and $[1,T]$, and using the Neumann boundary conditions on $V_i$, one gets
 \[
  \int_{\Omega}(V_i(\cdot,T)-V_i(\cdot,1))+\overline{C}\int_{\Omega}\int_1^TV_i\geq\beta_{HV}C\int_{\Omega}\int_1^TV,
 \]
 By the uniform Harnack inequality (see \cite{huska}, Theorem 2.5), it exists $C'>0$ independent of $T$ such that $\min_{\Omega}\int_1^TV_i\geq C'\max_{\Omega}\int_1^TV_i$. Then, one has, for all $x\in\Omega$,
 \[
  \begin{array}{ll}
   \overline{C}\int_0^TV_i(x,t)\mathrm{d}t & \geq\overline{C}\min_{\Omega}\int_1^TV_i\\
   & \geq\overline{C}C'\max_{\Omega}\int_1^TV_i\\
   & \geq\overline{C}C'|\Omega|\int_{\Omega}\int_1^TV_i\\
   & \geq\beta_{HV}CC'|\Omega|\int_{\Omega}\int_1^TV+\int_{\Omega}(V_i(\cdot,1)-V_i(\cdot,T)).
  \end{array}
 \]
 We have proved that, for an arbitrary $x\in\Omega$, if $V(x,\cdot)$ is not integrable at $+\infty$, then neither is $V_i(x,\cdot)$. It is then sufficient to show that the integral $\int_1^{+\infty}V(x,t)$ does not converge for any $x\in\Omega$.
 
 Let $\underline{V}=\underline{C}\dfrac{\varphi_{s,1}}{t+1}$, $\underline{C}=1/(2s_V\overline{\varphi_{s,1}})$. $\underline{V}$ solves
 \[
  \partial_t\underline{V}-\sigma_V\Delta\underline{V}=\left(r_V-h\dfrac{r_P}{s_P}\right)\underline{V}-\dfrac{1}{t+1}\underline{V}\leq\left(r_V-h\dfrac{r_P}{s_P}\right)\underline{V}-s_V\underline{V}^2
 \]
 and is then a subsolution of the Cauchy--Neumann system \eqref{preyeq} solved by $V$. We conclude that, for all $x\in\Omega$,
 \[
  \displaystyle\int_1^{+\infty}V(x,t)\mathrm{d}t\geq\int_1^{+\infty}\underline{V}(x,t)\mathrm{d}t=+\infty.
 \]
\end{proof}

We now deal with a simpler case of the limit of $\eta$ when $\lambda_1(\mathcal{L}_{V_s})$ approaches $0$ from the right: the case of $-r_V+h\dfrac{r_P}{s_P}$ uniformly converging to $0$. It is clear that it is a sufficient condition for $\lambda_1(\mathcal{L}_{V_s})$ converging to $0$, but it is not necessary.

\begin{proposition}
 Denote $m=-r_V+h\dfrac{r_P}{s_P}$. Consider the system \eqref{eqconstpred} with its associated harvest $\eta_m$. Let $(m_k)_{k\in\mathbb{N}}$ be a nonincreasing sequence (in the sense that $m_k\leq m_{k+1}$ for all $k\in\mathbb{N}$) of nonnegative functions that uniformly converges to $0$ as $k\to+\infty$. Suppose that $V_0\not\equiv 0$. Then,
 \[
  \lim_{k\to+\infty}{\eta_{m_k}}=0.
 \]
\end{proposition}

\begin{proof}
 The same proof as above yields that, for any $x\in\Omega$, if $V(x,\cdot)$ is not integrable, then $V_i(x,\cdot)$ is not integrable. With $\eta_{m_k}=\int_{\Omega}He^{-\beta_{VH}\int_0^{+\infty}V_i(\cdot,t)\mathrm{d}t}$, by the dominated convergence theorem, it suffices to prove that, for any $x\in\Omega$, $\lim_{k\to+\infty}{\int_0^{+\infty}V}=+\infty$.
 
 For all $k\in\mathbb{N}$, $V^k$ solves the equation
 \begin{equation}\label{preyeq2}
  \partial_tV^k-\sigma_V\Delta V^k =-m_kV^k-s_V(V^k)^2
 \end{equation}
 for $(x,t)\in\Omega\times\mathbb{R}_+^*$, with Neumann boundary conditions and initial condition $V(\cdot,0)=V_0$. By the comparison principle, and because $(m_k)$ is nonincreasing, $(V^k)$ is nondecreasing.
 
 By \cite{girardinmaucourt2023}, for all $\delta>0$, $\underset{x\in\Omega}{\inf}V^0(x,\delta)>0$. Denote $\underline{V}^k$ the solution of
 \[
  \begin{cases}
   \underline{V}^k(\cdot,1)=\underline{V}_1:=\underset{x\in\Omega}{\inf}\;(V^0)(x,1)>0 & \!\!\mbox{in }\Omega\\
   \partial_t\underline{V}^k-\sigma_V\Delta \underline{V}^k =-\overline{m}_k\underline{V}^k-s_V(\underline{V}^k)^2 & \!\!\mbox{in }\Omega\times[1,+\infty)\\
   \dfrac{\partial\underline{V}^k}{\partial n}=0 & \!\!\mbox{on }\partial\Omega\times[1,+\infty).
  \end{cases}
 \]
 which is a subsolution of the Cauchy--Neumann problem solved by $V^k$ on $\Omega\times[1,+\infty)$. We deduce, for all $t>1$,
 \[
  V^k(\cdot,t)\geq \underline{V}^k(t)=\dfrac{-\overline{m}_k}{s_V}\left(1+\left(\dfrac{-\overline{m}_k}{s_V\underline{V}_1}-1\right)e^{\overline{m}_kt}\right)^{-1}.
 \]
 
 By assumption, $\overline{m}_k:=\Vert m_k\Vert_{\infty}$ converges to $0$ when $k\to+\infty$. For all $x\in\Omega$,
 \[
  \begin{array}{ll}
   \displaystyle\int_0^{+\infty}V^k(x,t)\mathrm{d}t & \geq\dfrac{\overline{m}_k}{s_V}\displaystyle\int_1^{+\infty}\left(-1+\left(\dfrac{\overline{m}_k}{s_V\underline{V}_1}+1\right)e^{\overline{m}_kt}\right)^{-1}\mathrm{d}t\\[10pt]
   & =\dfrac{\overline{m}_k}{s_V}\displaystyle\int_1^{+\infty}e^{-\overline{m}_kt}\left(\dfrac{\overline{m}_k}{s_V\underline{V}_1}+1-e^{-\overline{m}_kt}\right)^{-1}\mathrm{d}t\\[10pt]
   & =\dfrac{1}{s_V}\left[\ln\left(\dfrac{\overline{m}_k}{s_V\underline{V}_1}+1-e^{-\overline{m}_kt}\right)\right]_1^{+\infty}\\[10pt]
   & =\dfrac{1}{s_V}\ln\left(\dfrac{\overline{m}_k/(s_V\underline{V}_1)+1}{\overline{m}_k/(s_V\underline{V}_1)+1-e^{-\overline{m}_k}}\right)\underset{k\to+\infty}{\longrightarrow}+\infty.
  \end{array}
 \]
\end{proof}
 
\section{Homogenization of the refuges}

In this section, we prove our primary homogenization result. In order to define properly a ``refuge of higher frequency'', we assume that $\Omega$ is a hyperrectangle in $\mathbb{R}^d$. In the example of the beet field, one can take $d=2$.

We can periodically extend a refuge $R\in \mathrm{L}^{\infty}(\Omega,[0,1])$ in $\mathbb{R}^d$ and define $R_n:x\mapsto R(nx)$ the refuge of frequency $n$.

For a function $f\in\mathrm{L}^{\infty}(\Omega)$, we denote
\[
 \left\langle f\right\rangle_{\Omega}=\dfrac{1}{|\Omega|}\int_{\Omega}f
\]
the average of $f$ over $\Omega$.

Denote $\mathcal{R}:=\left\langle R\right\rangle_{\Omega}$ and $\mathcal{R}_2:=\left\langle R^2\right\rangle_{\Omega}$. One can easily verify that for all $n\in\mathbb{N}^*$, $\left\langle R_n\right\rangle_{\Omega}=\mathcal{R}$ and $\left\langle R_n^2\right\rangle_{\Omega}=\mathcal{R}_2$.

For such a refuge, we denote
\[
 \begin{array}{l}
  H_n=H^0(1-R_n)\\
  b_V^n=(b_V^r-b_V^f)R_n+b_V^f\\
  d_V^n=(d_V^r-d_V^f)R_n+d_V^f\\
  r_V^n=b_V^n-d_V^n\\
  r_P^n=(r_P^r-r_P^f)R_n+r_P^f.
 \end{array}
\]

We consider a similar Cauchy--Neumann problem as \eqref{eqspring} with a refuge $R_n$, for which the system rewrites itself
\begin{equation}\label{systemfreqn}
 \begin{cases}
 \partial_tI & =\beta_{VH}(H_n-I)V_i \\
 \partial_tV_i & =\sigma_V\Delta V_i+\beta_{HV}IV_s-(\alpha+d_V^n+s_V(V_s+V_i)+hP)V_i \\
 \partial_tV_s & =\sigma_V\Delta V_s+(\alpha+b_V^n)V_i+(r_V^n-\beta_{HV}I-s_V(V_s+V_i)-hP)V_s \\
 \partial_t P & =\sigma_P\nabla\cdot\left(r_P^n\nabla\left(\dfrac{P}{r_P^n}\right)\right)+(\gamma h(V_s+V_i)+r_P^n-s_PP)P.
 \end{cases}
\end{equation}
with nonnegative initial condition $(I, V_i, V_s, P)(\cdot,0) = (I_0^n, V_{i,0}^n, V_{s,0}^n, P_0^n)\quad\text{in }\Omega$.

Denote $(I_n,V_i^n,V_s^n,P_n)$, the unique solution of this new problem, and
\[
 \eta_n=\left(\int_{\Omega}H-I_{\infty}\right)_n=\int_{\Omega}H_ne^{-\beta_{VH}\int_0^{+\infty}V_i^n}
\]
the associated harvest.

We also denote
\[
 \begin{array}{rll}
  H_{\infty} & =\left\langle H_n\right\rangle_{\Omega} & =H^0(1-\mathcal{R})\\
  b_V^{\infty} & =\left\langle b_V^n\right\rangle_{\Omega} & =(b_V^r-b_V^f)\mathcal{R}+b_V^f\\
  d_V^{\infty} & =\left\langle d_V^n\right\rangle_{\Omega} & =(d_V^r-d_V^f)\mathcal{R}+d_V^f\\
  r_V^{\infty} & =\left\langle r_V^n\right\rangle_{\Omega} & =b_V^{\infty}-d_V^{\infty}\\
  r_P^{\infty} & =\left\langle r_P^n\right\rangle_{\Omega} & =(r_P^r-r_P^f)\mathcal{R}+r_P^f\\
  (r_P^{\infty})_2 & =\left\langle (r_P^n)^2\right\rangle_{\Omega} & =(r_P^r-r_P^f)^2\mathcal{R}_2+2(r_P^r-r_P^f)r_P^f\mathcal{R}+(r_P^f)^2,
 \end{array}
\]
and $r^*\in\mathcal{M}_d(\mathbb{R})$ the so-called ``homogenized conductivity'' defined in \cite{allaire2010}.

The homogenized Cauchy--Neumann problem is the following:
\begin{equation}\label{systemlimit}
 \begin{cases}
 \partial_tI & =\beta_{VH}(H_{\infty}-I)V_i \\
 \partial_tV_i & =\sigma_V\Delta V_i+\beta_{HV}IV_s-(\alpha+d_V^{\infty}+s_V(V_s+V_i)+hP)V_i \\
 \partial_tV_s & =\sigma_V\Delta V_s+(\alpha+b_V^{\infty})V_i+(r_V^{\infty}-\beta_{HV}I-s_V(V_s+V_i)-hP)V_s \\
 \partial_t P & =\sigma_P\nabla\cdot\left(r^*\nabla\left(\dfrac{P}{r_P^{\infty}}\right)\right)+(\gamma h(V_s+V_i)+\dfrac{(r_P^{\infty})_2}{r_P^{\infty}}-s_P\dfrac{(r_P^{\infty})_2}{(r_P^{\infty})^2}P)P.
 \end{cases}
\end{equation}
with nonnegative initial condition $(I, V_i, V_s, P)(\cdot,0) = (I_0^{\infty}, V_{i,0}^{\infty}, V_{s,0}^{\infty}, P_0^{\infty})\quad\text{in }\Omega$. We assume the weak convergence in $\mathrm{L}^2(\Omega)$ of $I_0^n, V_{i,0}^n, V_{s,0}^n, P_0^n$ to $I_0^{\infty}, V_{i,0}^{\infty}, V_{s,0}^{\infty}, P_0^{\infty}$, respectively.

Denote $(I^{\infty},V_i^{\infty},V_s^{\infty},P^{\infty})$, the unique solution of this problem, and
\[
 \eta_{\infty}=\left(\int_{\Omega}H-I_{\infty}\right)_{\infty}=\int_{\Omega}H_{\infty}e^{-\beta_{VH}\int_0^{+\infty}V_i^{\infty}}
\]
the associated harvest.

\textbf{Remark:} The classical homogenization techniques, as developed by Murat and Tartar \cite{murattartar1977}, elaborated upon in books like \cite{bakhvalov1989}, and well comprehensively explained by Allaire \cite{allaire2010}, are applicable in this context, particularly for the fourth equation in \eqref{systemfreqn}. As for the first three equations in \eqref{systemfreqn}, given that the parameter $n$ solely appears in the reaction term, straightforward estimates lead to the strong convergence in $\mathrm{L}^2(\Omega\times[0,T])$ of the solutions toward those of the homogenized problem. However, the presence of $r_P^n$ in the divergence term of the fourth equation in \eqref{systemfreqn} hinders a straightforward limit transition without further considerations. In fact, unlike the sequences $(I_n)$, $(V_i^n)$, and $(V_s^n)$, the sequence $(P_n)$ converges only weakly in $\mathrm{L}^2(\Omega\times[0,T])$.

\begin{lemma}
 The sequences $(H_n)$, $(b_V^n)$, $(d_V^n)$, $(r_V^n)$ and $(r_P^n)$ weakly converge to, respectively, $H_{\infty}$, $b_V^{\infty}$, $d_V^{\infty}$, $r_V^{\infty}$ and $r_P^{\infty}$ in $\mathrm{L}^2(\Omega)$.
\end{lemma}

\begin{proof}
 The weak convergence of the five sequences comes from the weak convergence of the sequence $(R_n)$ to $x\mapsto\mathcal{R}$. Indeed, say $\Omega=[0,L]$ and let $\mathbf{1}_{[a,b]}$ be an indicator function. A change of variable yields
 \[
  \int_a^bR_n=\dfrac{1}{n}\int_{na}^{nb}R.
 \]
 
 For $k\in\mathbb{N}^*$, if $kL\leq n(b-a)\leq (k+1)L$, the segment $[na,nb]$ intersects at least $k$ intervals of size $L$, and at best $k+1$, therefore
 \[
  \dfrac{k}{n}L\mathcal{R}\leq\dfrac{1}{n}\int_{na}^{nb}R\leq\dfrac{k+1}{n}L\mathcal{R},
 \]
 thus
 \[
  \left(b-a-\dfrac{L}{n}\right)\mathcal{R}\leq\dfrac{1}{n}\int_{na}^{nb}R\leq\left(b-a+\dfrac{L}{n}\right)\mathcal{R}
 \]
 and
 \[
  \dfrac{1}{n}\int_{na}^{nb}R\underset{n\to+\infty}{\longrightarrow}(b-a)\mathcal{R}=\int_{\Omega}\mathbf{1}_{[a,b]}\mathcal{R}.
 \]
 
 With a very similar proof, this is also true when $\Omega$ is a rectangle in $\mathbb{R}^d$, and trivially for any step function. We deduce the weak convergence of $(R_n)$ by density of the space of step functions in $\mathrm{L}^2(\Omega)$.
\end{proof}

We then need to assess the strong convergence almost everywhere in $\Omega\times[0,T]$ for $T>0$ of the solutions when $n\to+\infty$.

\begin{lemma}\label{strongCVofVi}
 Let $T>0$ be arbitrary. Consider the Cauchy--Neumann problem \eqref{systemfreqn} for $(x,t)\in\Omega\times[0,T]$, and call, for $n\in\mathbb{N}$, $(I_n,V_i^n,V_s^n,P_n)$ the unique solution. Then, the sequences $(I_n)$, $(V_i^n)$, $(V_s^n)$ and $(P_n/r_P^n)$ strongly converge in $\mathrm{L}^2(\Omega\times[0,T])$ and almost everywhere in $\Omega\times[0,T]$ to, respectively, $(I_{\infty})$, $(V_i^{\infty})$, $(V_s^{\infty})$ and $(P_{\infty}/r_P^{\infty})$, and $P_{\infty}$ is the weak limit in $\mathrm{L}^2(\Omega\times[0,T])$ of $(P_n)$.
\end{lemma}
 
\begin{proof}
 This proof is very much inspired by \cite{dancer1999}.
 
 First of all, by \cite{girardinmaucourt2023}, for all $n\in\mathbb{N}$, $(I_n,V_i^n,V_s^n,P_n/r_P^n)$ are bounded in $\mathrm{L}^{\infty}(\Omega\times[0,T])$ by a constant $C>0$ independent of $n$.
 
 Multiplying the second equation by $V_i$ and integrating over $\Omega$, one gets
 \[
 \begin{array}{lll}
  \dfrac{1}{2}\dfrac{\partial}{\partial t}\int_{\Omega}(V_i^n)^2+\sigma_V\int_{\Omega}|\nabla V_i^n|^2 & = & \int_{\Omega}\beta_{HV}I_nV_s^nV_i^n\\
  && -\int_{\Omega}(\alpha+d_V+s_V(V_s^n+V_i^n)+hP_n)(V_i^n)^2\\
  & \leq & |\Omega|\beta_{HV}C^3.
 \end{array}
 \]
 
 Integrating this over $[0,T]$ yields the existence of a constant $C'>0$ independent of $n$ such that
 \[
  \int_0^T\int_{\Omega}|\nabla V_i^n|^2\leq C',
 \]
 hence the boundedness of the $\mathrm{L}^2([0,T],\mathrm{H}^1(\Omega))$-norm of $V_i^n$.
 
 Similarly, multiplying the first equation by $I$ and integrating over $\Omega$ and $[0,T]$, multiplying the third equation by $V_s$ and integrating over $\Omega$ and $[0,T]$, and multiplying the fourth equation by $P/r_P$ and integrating over $\Omega$ and $[0,T]$ yields the boundedness of the $\mathrm{L}^2([0,T],\mathrm{H}^1(\Omega))$-norm of $I_n$, $V_s^n$ and $P_n/r_P^n$.
 
 Let $\zeta\in \mathrm{L}^2([0,T],\mathrm{H}^1(\Omega))$ be arbitrary. Multiplying the second equation by $\zeta$ and integrating over $\Omega\times[0,T]$, one gets
 {\small\[
  \int_0^T\!\int_{\Omega}\partial_tV_i^n\zeta+\sigma_V\!\int_0^T\!\int_{\Omega}\nabla V_i^n\cdot\nabla\zeta=\int_0^T\!\int_{\Omega}(\beta_{HV}I_nV_s^n-(\alpha+d_V+s_V(V_s^n+V_i^n)+hP_n)V_i^n)\zeta,
 \]}
 hence
 \[
 \begin{array}{ll}
  \left|\int_0^T\int_{\Omega}\partial_tV_i^n\zeta\right| & \leq\sigma_V\sqrt{C'}\sqrt{\int_0^T\int_{\Omega}|\nabla\zeta|^2}+C''\sqrt{\int_0^T\int_{\Omega}\zeta^2}\\
  & \leq \max(\sigma\sqrt{C'},C'')\Vert\zeta\Vert_{\mathrm{L}^2([0,T],\mathrm{H}^1(\Omega))},
 \end{array}
 \]
 
 where $C''$ is a positive constant independent of $n$. By the arbitrariness of $\zeta$, this shows that $\Vert\partial_tV_i^n\Vert_{\mathrm{L}^2([0,T],(\mathrm{H}^1(\Omega))')}\leq\max(\sigma\sqrt{C},C'')$.
 
 Similarly, multiplying the first, second and third equation by $\zeta$ and integrating over $\Omega\times[0,T]$ yields bounds for the $\mathrm{L}^2([0,T],(\mathrm{H}^1(\Omega))')$-norm of $\partial_tI_n$, $\partial_tV_s^n$ and $\partial_t(P_n/r_P^n)$ (recall that $r_P$ is independent of time).
 
 By the Aubin-Lions theorem, the sequences $(I_n)$, $(V_i^n)$ $(V_s^n)$ and $(P_n/r_P^n)$ are precompact in $\mathrm{L}^2(\Omega\times[0,T])$.
 
 Then, up to extraction, we obtain strong convergence of the four sequences in $\mathrm{L}^2(\Omega\times[0,T])$ and almost everywhere in $\Omega\times[0,T]$. Denote $\overline{I}$, $\overline{V_i}$, $\overline{V_s}$ and $\overline{P_r}$ the respective limits. In the following, we denote again by $(I_n)$, $(V_i^n)$ $(V_s^n)$ and $(P_n/r_P^n)$ the extracted sequences. Recall that the initial conditions converge weakly in $\mathrm{L}^2(\Omega)$ of the initial conditions of the homogenized system.
 
 Let $\varphi\in\mathcal{C}^2(\Omega\times[0,T])$ verifying the Neumann boundary conditions in $\Omega$ and such that $\varphi(\cdot,T)=0$. Multiplying the first equation of \eqref{systemfreqn} by $\varphi$ and integrating over $\Omega\times[0,T]$ yields
 \begin{equation}\label{eq3varphi}
  \int_0^T\int_{\Omega}\varphi\partial_t I_n=\beta_{VH}\int_0^T\int_{\Omega}\varphi(H_n-I_n)V_i^n.
 \end{equation}
 
 Moreover,
 \[
  \begin{array}{ll}
   \int_0^T\int_{\Omega}\varphi\partial_t I_n & =\int_{\Omega}\left(\left[\varphi I_n\right]_0^T-\int_0^T\partial_t\varphi I_n\right)\\
   & =-\int_{\Omega}\left(\varphi(\cdot,0)I_0+\int_0^T\partial_t\varphi I_n\right),
  \end{array}
 \]
 
 \[
  \begin{array}{ll}
   & \left|\int_0^T\int_{\Omega}\partial_t\varphi I_n-\int_0^T\int_{\Omega}\partial_t\varphi \overline{I}\right|\\
   \leq & \left|\int_{\Omega}\varphi(\cdot,0)(I_0^{\infty}-I_0^n)\right|+\sqrt{\int_0^T\int_{\Omega}(\partial_t\varphi)^2}\Vert I_n-I\Vert_{\mathrm{L}^2(\Omega\times[0,T])}\\
   & \underset{n\to+\infty}{\longrightarrow}0,
  \end{array}
 \]

 {\small\[
 \begin{array}{lll}
  \left|\int_0^T\int_{\Omega}\varphi H_nV_i^n-\int_0^T\int_{\Omega}\varphi H_{\infty}\overline{V_i}\right| & = & \left|\int_0^T\int_{\Omega}\varphi H_n(V_i^n-\overline{V_i})+\int_0^T\int_{\Omega}\varphi(H_n-H_{\infty})\overline{V_i}\right|\\
  & \leq & \sqrt{\int_0^T\int_{\Omega}\varphi^2H_n^2}\Vert V_i^n-\overline{V_i}\Vert_{\mathrm{L}^2(\Omega\times[0,T])}\\
  && +\left|\int_0^T\int_{\Omega}(H_n-H_{\infty})\varphi \overline{V_i}\right|\\
  && \underset{n\to+\infty}{\longrightarrow}0
 \end{array}
 \]}
 and
 {\small\[
 \begin{array}{ll}
  \left|\int_0^T\int_{\Omega}\varphi I_nV_i^n-\int_0^T\int_{\Omega}\varphi\overline{I}\overline{V_i}\right| & =\left|\int_0^T\int_{\Omega}\varphi(I_n-\overline{I})V_i^n+\int_0^T\int_{\Omega}\varphi\overline{I}(V_i^n-\overline{V_i})\right|\\
  & \leq T\max(r_V^r,r_V^f)/s_V\sqrt{\int_0^T\int_{\Omega}\varphi^2}\Vert I_n-I\Vert_{\mathrm{L}^2(\Omega\times[0,T])}\\
  & +\sqrt{\int_0^T\int_{\Omega}\varphi^2\overline{I}^2}\Vert V_i^n-\overline{V_i}\Vert_{\mathrm{L}^2(\Omega\times[0,T])}\\
  & \underset{n\to+\infty}{\longrightarrow}0.
 \end{array}
 \]}
 
 Therefore, letting $n$ go to $+\infty$ in \eqref{eq3varphi}, one gets
 \[
  \int_0^T\int_{\Omega}\varphi\partial_t\overline{I}=\beta_{VH}\int_0^T\int_{\Omega}\varphi(H_{\infty}-\overline{I})\overline{V_i}.
 \]
 
 Similarly, multiplying the second and third equation of \eqref{systemfreqn} by $\varphi$, and letting $n$ go to $+\infty$, also using the following convergence:
 \[
 \begin{array}{ll}
  \left|\int_0^T\int_{\Omega}\varphi\Delta V_i^n-\int_0^T\int_{\Omega}\varphi\Delta\overline{V_i}\right| & =\left|\int_0^T\int_{\Omega}\Delta\varphi(V_i^n-\overline{V_i})\right|\\
  & \leq\sqrt{\int_0^T\int_{\Omega}(\Delta\varphi)^2}\Vert V_i^n-\overline{V_i}\Vert_{\mathrm{L}^2(\Omega\times[0,T])}\\
  & \underset{n\to+\infty}{\longrightarrow}0,
 \end{array}
 \]
 one gets:
 \[
  \int_0^T\int_{\Omega}\varphi\partial_t\overline{V_i}=\int_0^T\int_{\Omega}\varphi(\sigma_V\Delta \overline{V_i}+\beta_{HV}\overline{I}\overline{V_s}-(\alpha+d_V^{\infty}+s_V(\overline{V_s}+\overline{V_i})+h\overline{P})\overline{V_i})
 \]
 and
 \[
  \int_0^T\int_{\Omega}\varphi\partial_t\overline{V_s}=\int_0^T\int_{\Omega}\varphi(\sigma_V\Delta \overline{V_s}+(\alpha+b_V^{\infty})\overline{V_i}+(r_V^{\infty}-\beta_{HV}\overline{I}-s_V(\overline{V_s}+\overline{V_i})-h\overline{P})\overline{V_s})
 \]
 (where $\overline{P}$ is the weak limit of $P_n$, whose existence will be justified hereafter). We deal with the fourth equation by applying the change of variable $\tilde{P}=P/r_P^n$:
 \[
  r_P^n\partial_t \tilde{P}=\sigma_P\nabla\cdot\left(r_P^n\nabla\tilde{P}\right)+(\gamma h(V_s+V_i)+r_P^n-s_Pr_P^n\tilde{P})r_P^n\tilde{P}
 \]
and by carrying out a classical two-scale asymptotic expansion (also called
an ansatz, cf \cite{allaire2010}), we find that $(P_n/r_P^n)$ converges weakly in $\mathrm{L}^2([0,T],\mathrm{H}^1(\Omega))$ to the solution (in the weak sense) of
 \begin{equation}\label{eqpredhomo}
  r_P^{\infty}\partial_t \tilde{P}=\sigma_P\nabla\cdot\left(r^*\nabla\tilde{P}\right)+(\gamma hr_P^{\infty}(\overline{V_s}+\overline{V_i})+(r_P^{\infty})_2-s_P(r_P^{\infty})_2\tilde{P})\tilde{P}.
 \end{equation}
 
 We voluntarily elude the computations to lighten the proof, as they are almost identical to the examples given in \cite{allaire2010}.
 
 By uniqueness of the limit, $\overline{P_r}$ solves equation \eqref{eqpredhomo}. Moreover,
 \[
  P_n=r_P^n(P_n/r_P^n)\underset{n\to+\infty}{\longrightarrow}r_P^{\infty}\overline{P_r}=:\overline{P}\mbox{ weakly in }\mathrm{L}^2(\Omega\times[0,T]).
 \]
 
 By plugging $\overline{P}$ in \eqref{eqpredhomo}, one gets (in the weak sense)
 \[
  \partial_t \overline{P}=\sigma_P\nabla\cdot\left(r^*\nabla\left(\dfrac{\overline{P}}{r_P^{\infty}}\right)\right)+(\gamma h(\overline{V_s}+\overline{V_i})+\dfrac{(r_P^{\infty})_2}{r_P^{\infty}}-s_P\dfrac{(r_P^{\infty})_2}{(r_P^{\infty})^2}\overline{P})\overline{P}.
 \]
 
 The weak Neumann boundary condition follows easily from the weak convergence of the gradient of the sequences.

 By uniqueness of the weak solution of \eqref{systemlimit}, $(\overline{I},\overline{V_i},\overline{V_s},\overline{P})=(I_{\infty},V_i^{\infty},V_s^{\infty},P_{\infty})$, and we obtain the strong convergence in $\mathrm{L}^2(\Omega\times[0,T])$ and almost everywhere in $\Omega\times[0,T]$ of the non-extracted sequences $(I_n)$, $(V_i^n)$, $(V_s^n)$ and $(P_n/r_P^n)$ and the weak convergence in $\mathrm{L}^2(\Omega\times[0,T])$ of the non-extracted sequence $(P_n)$.

\end{proof}

\begin{theorem}[Theorem~\ref{thhomogen}]
 Suppose that $\lambda_1(\mathcal{L}_{V_s})>0$.
 
 The harvest $\eta_n=\left(\displaystyle\int_{\Omega}H-I_{\infty}\right)_n$ converges to $\eta_{\infty}=\left(\displaystyle\int_{\Omega}H-I_{\infty}\right)_{\infty}$ as $n\to+\infty$.
\end{theorem}

\begin{proof}
 For all $T>0$, by Lemma \ref{strongCVofVi}, the sequence $(V_i^n)$ strongly converges in $\mathrm{L}^2(\Omega\times[0,T])$ to $V_i^{\infty}$. By positivity of $\lambda_1(\mathcal{L}_{V_s})$, since $n\mapsto\lambda_1^n(\mathcal{L}_{V_s})$ is increasing and converges to the principal eigenvalue associated to the constant refuge $\mathcal{R}$ we denote $\lambda_1^{\infty}(\mathcal{L}_{V_s})$, by \cite{girardinmaucourt2023}, for a small enough $\varepsilon>0$, there exists a constant $C>0$ and $t_0>0$ such that, for all $x\in\Omega$, $t\geq t_0$ and $n\in\mathbb{N}$ large enough,
 \[
  V_i^n(x,t)\leq Ce^{-(\lambda_1^n(\mathcal{L}_{V_s})-h\varepsilon)(t-t_0)}\leq Ce^{-(\lambda_1(\mathcal{L}_{V_s})-h\varepsilon)(t-t_0)}
 \]
 and
 \[
  V_i^{\infty}(x,t)\leq Ce^{-(\lambda_1^{\infty}(\mathcal{L}_{V_s})-h\varepsilon)(t-t_0)}\leq Ce^{-(\lambda_1(\mathcal{L}_{V_s})-h\varepsilon)(t-t_0)}.
 \]
 
 Let $x\in\Omega$, $T\geq t_0$ such that $\dfrac{2C|\Omega|}{\lambda_1(\mathcal{L}_{V_s})-h\varepsilon}e^{-(\lambda_1(\mathcal{L}_{V_s})-h\varepsilon)(T-t_0)}\leq\varepsilon/2$ and $n\in\mathbb{N}$ verifying the above inequality, and such that $\Vert V_i^n-V_i^{\infty}\Vert_{\mathrm{L}^2(\Omega\times[0,T])}^2\leq\dfrac{\varepsilon}{2|\Omega|T}$.
 \[
 \begin{array}{ll}
  & \displaystyle\int_{\Omega}\left|\int_0^{+\infty}(V_i^n(x,t)-V_i^{\infty}(x,t))\mbox{d}t\right|\mbox{d}x\\
  \leq & \displaystyle\int_{\Omega}\left|\int_0^T(V_i^n(x,t)-V_i^{\infty}(x,t))\mbox{d}t\right|\mbox{d}x+\displaystyle\int_{\Omega}\left|\int_T^{+\infty}(V_i^n(x,t)-V_i^{\infty}(x,t))\mbox{d}t\right|\mbox{d}x\\
  \leq & \Vert V_i^n-V_i^{\infty}\Vert_{\mathrm{L}^2(\Omega\times[0,T])}^2|\Omega|T+2C\displaystyle\int_{\Omega}\int_T^{+\infty}e^{-(\lambda_1(\mathcal{L}_{V_s})-h\varepsilon)(t-t_0)}\mbox{d}t\mbox{d}x\\
  \leq & \varepsilon/2+\dfrac{2C|\Omega|}{\lambda_1(\mathcal{L}_{V_s})-h\varepsilon}e^{-(\lambda_1(\mathcal{L}_{V_s})-h\varepsilon)(T-t_0)}\\
  \leq & \varepsilon.
 \end{array}
 \]
 
 We have shown the convergence of $\int_0^{+\infty}V_i^n(\cdot,t)\mbox{d}t$ to $\int_0^{+\infty}V_i^{\infty}(\cdot,t)\mbox{d}t$ in $\mathrm{L}^1(\Omega)$. Up to extraction, the sequence $(\int_0^{+\infty}V_i^n)$ hereby converges almost everywhere in $\Omega$ to $\int_0^{+\infty}V_i^{\infty}$. By positivity of the sequence, and by the dominated convergence theorem, we obtain the convergence of $e^{-\beta_{VH}\int_0^{+\infty}V_i^n}$ to $e^{-\beta_{VH}\int_0^{+\infty}V_i^{\infty}}$ in $\mathrm{L}^1(\Omega)$.
 
 Let $n\in\mathbb{N}$ be large enough so that $\Vert e^{-\beta_{VH}\int_0^{+\infty}V_i^n}-e^{-\beta_{VH}\int_0^{+\infty}V_i^{\infty}}\Vert_{\mathrm{L}^1(\Omega)}\leq\varepsilon/(2H^0)$ and $\left|\displaystyle\int_{\Omega}(H_n-H_{\infty})e^{-\beta_{VH}\int_0^{+\infty}V_i^{\infty}}\right|\leq\varepsilon/2$.
 \[
 \begin{array}{ll}
  & \left|\displaystyle\int_{\Omega}H_ne^{-\beta_{VH}\int_0^{+\infty}V_i^n}-\displaystyle\int_{\Omega}H_{\infty}e^{-\beta_{VH}\int_0^{+\infty}V_i^{\infty}}\right|\\
  = & \left|\displaystyle\int_{\Omega}H_n(e^{-\beta_{VH}\int_0^{+\infty}V_i^n}-e^{-\beta_{VH}\int_0^{+\infty}V_i^{\infty}})+\displaystyle\int_{\Omega}(H_n-H_{\infty})e^{-\beta_{VH}\int_0^{+\infty}V_i^{\infty}}\right|\\
  \leq & H^0\displaystyle\int_{\Omega}|e^{-\beta_{VH}\int_0^{+\infty}V_i^n}-e^{-\beta_{VH}\int_0^{+\infty}V_i^{\infty}}|+\left|\displaystyle\int_{\Omega}(H_n-H_{\infty})e^{-\beta_{VH}\int_0^{+\infty}V_i^{\infty}}\right|\\
  \leq & \varepsilon.
 \end{array}
 \]
 
 Hence the convergence of $\eta_n$ to $\eta_{\infty}$.

\end{proof}

\section{Optimizers of the linearized harvest}

We devote the next section to the computation of what we call the linearized harvest $\eta_L$. We would like to justify the necessity of utilizing a linearized version of $\eta$ by briefly examining a straightforward scenario within our problem, wherein all spatial variables remain constant and specific parameter values are applied. Even in this case, the explicit computation of the solution for $V$ (let alone $V_i$) remains unattainable. When we combine the equations for $V_i$ and $V_s$, we obtain the following prey--predator system.

\begin{equation}\label{preypredsyst}
 \begin{cases}
  \partial_tV-\sigma_V\Delta V =r_V(x)V-s_VV^2-hPV\\
  \partial_t P-\sigma_P\overline{\mathcal{L}}(P)=r_P(x)P-s_PP^2+\gamma hPV
 \end{cases}
\end{equation}

Now, assume that $R$ is constant in space, so that $r_V$ and $r_P$ are constant, and assume that $V_0$ and $P_0$ are constant. To simplify even further, assume that $r_V=r_P=s_V=s_P=1$ and $(\gamma-1)h=-2$ (for instance, $\gamma=1/2$ and $h=4$). The system now writes:

\begin{equation}\label{preypredsystconst}
 \begin{cases}
  V'=V-V^2-hPV\\
  P'=P-P^2+\gamma hPV
 \end{cases}
\end{equation}

Denote $u=V+P$. $u$ solves
\[
 \begin{array}{ll}
  u' & =u-(V^2+P^2)+(\gamma-1)hPV\\
  & =u-u^2+2PV+(\gamma-1)hPV\\
  & =u-u^2,
 \end{array}
\]
which yields, for all $t>0$,
\[
 u(t)=\dfrac{1}{1+\left(\dfrac{1}{V_0+P_0}-1\right)e^{-t}},
\]
which means $V$ solves
\[
 \begin{array}{ll}
  V' & =V-V^2-h(u-V)V\\
  & =(1-hu)V+(h-1)V^2.
 \end{array}
\]

Although this equation is entirely decorrelated with the variable $P$, it remains too hard to compute explicitly or to calculate its time integral from $0$ to $+\infty$. As a result, we lack information regarding the derivative of $V$ with respect to $R$, and thus, we also lack information on the derivative of $\eta$ with respect to $R$. To the best of our knowledge, this renders the nonlinear optimization problem too complicated to address directly via $\eta$. However, we can assume that $R$ is constant, $V_0$ and $P_0=r_P/s_P$ constant, and $\gamma=0$, so that $P(t)=r_P/s_P$ for all $t>0$, hence $V$ solves
\[
 V'=(r_V-h\dfrac{r_P}{s_P})V-s_VV^2
\]
and one can compute, denoting $\hat{m}=hr_P/s_P-r_V$, for all $t>0$
\[
 V(t)=\dfrac{\hat{m}}{s_V}\left(\left(\dfrac{\hat{m}}{s_VV_0}+1\right)e^{\hat{m}t}-1\right)^{-1}
\]
and, if $\hat{m}$ is positive,
\[
 \begin{array}{ll}
  \displaystyle\int_0^{+\infty}V(t)\mathrm{d}t & =\dfrac{1}{s_V}\left[\ln\left(\dfrac{\hat{m}}{s_VV_0}+1-e^{-mt}\right)\right]_0^{+\infty}\\
  & =\dfrac{1}{s_V}\ln\left(1+\dfrac{s_VV_0}{\hat{m}}\right),
 \end{array}
\]
hence a sub-estimation of the harvest, that we call $\eta_V$:
\[
 \begin{array}{ll}
  \eta_V & =\displaystyle\int_{\Omega}He^{-\beta_{VH}\int_0^{+\infty}V}\\
  & =H|\Omega|\left(1+\dfrac{s_VV_0}{\hat{m}}\right)^{-\beta_{VH}/s_V}.
 \end{array}
\]

We can make explicit the $R$-dependency by plugging the expressions of $H$ and $\hat{m}$ in function of $R\in[0,1]$:
\[
 \eta_V=H^0|\Omega|(1-R)\left(1+\dfrac{s_VV_0}{\tilde{\xi}+\tilde{m}R}\right)^{-\beta_{VH}/s_V},
\]
where $\tilde{\xi}=hr_P^f/s_P-r_V^f$ and $\tilde{m}=h(r_P^r-r_P^f)/s_P-r_V^r+r_V^f$. We plot this function in Figure \ref{subharvestfunofR}.

\begin{center}
 \begin{figure}
  \includegraphics[scale=0.8]{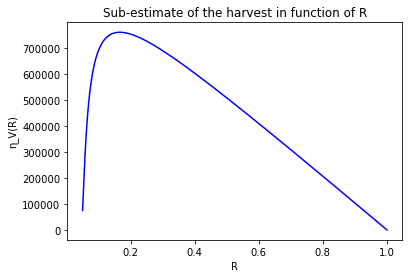}
  \caption{Sub-estimate of the harvest in function of a constant $R$ (with biologically consistent parameters, see \cite{girardinmaucourt2023}).}
  \label{subharvestfunofR}
 \end{figure}
\end{center}

A simple computation yields the optimal $R$ of $\eta_V$ we call $R_{\textnormal{opt}}^V$:
\[
 \begin{array}{ll}
  R_{\textnormal{opt}}^V= & \dfrac{1}{2\tilde{m}^2}\left(\sqrt{\tilde{m}^2(2\tilde{\xi}+s_VV_0-\beta_{VH}V_0)^2-4\tilde{m}^2(\tilde{\xi}^2+s_VV_0\tilde{\xi}-\beta_{VH}\tilde{m}V_0)}\right.\\
  & +\tilde{m}\left.(\beta_{VH}V_0-2\tilde{\xi}-s_VV_0)\right).
 \end{array}
\]

Here, $R_{\textnormal{opt}}^V\approx0,19$.

\textbf{Comment:} By doing a Taylor expansion of $\eta_V$ around $V_0=0$, one gets
\[
 \eta_V=H^0|\Omega|(1-R)\left(1-\dfrac{\beta_{VH}V_0}{\tilde{\xi}+\tilde{m}R}\right)+o(V_0),
\]
which is an expression similar to that of the linearized harvest in the constant case, that will be computed later on in Remark 1.

\subsection{Computation of $\eta_L$}

Let us now delve into the linearization. We suppose in the following that $\lambda_1(\mathcal{L}_{V_s})$, the principal eigenvalue of $-\sigma_V\Delta+\dfrac{h}{s_P}r_P-r_V$ is positive, so that $V$ converges to $0$ exponentially fast uniformly in $\Omega$ as $t\to+\infty$, as well as the exponential convergence of $P$ to $r_P/s_P$ (see \cite{girardinmaucourt2023}). This ensures the boundedness of $V$, $P$ in $\Omega\times\mathbb{R}_+^*$, and the existence of $\displaystyle\int_{\Omega}\int_0^{+\infty}\max(V,P-r_P/s_P)$.

Let us define some useful constants to lighten the writing:
\begin{itemize}
 \item $m=\dfrac{h}{s_P}(r_P^r-r_P^f)+d_V^r-d_V^f$,
 \item $\xi=\alpha+d_V^f+\dfrac{h}{s_P}r_P^f>0$.
\end{itemize}

We suppose in the following that $m$ is positive. This will ensure that placing refuges in the field increases $\lambda_1(\mathcal{L}_{V_s})$, and therefore is beneficial for a quicker eradication of the population $V$. In the case $m\leq 0$, the optimisation problem is trivial: $R=0$ yields the best harvest (which is $0$ in the case $\lambda_1(\mathcal{L}_{V_s})<0$).

Denote $\tilde{P}=P-r_P/s_P$ and
\[
 o(\Vert(V,\tilde{P})\Vert)=\underset{\Vert(V,\tilde{P})\Vert_{\mathrm{L}^{\infty}(\Omega\times\mathbb{R}_+^*)}\to0}{o}\left(\Vert(V,\tilde{P})\Vert_{\mathrm{L}^{\infty}(\Omega\times\mathbb{R}_+^*)}\right).
\]

We have
\[
 \dfrac{\Vert IV_s\Vert_{\mathrm{L}^{\infty}(\Omega\times\mathbb{R}_+^*)}}{\Vert(V,\tilde{P})\Vert_{\mathrm{L}^{\infty}(\Omega\times\mathbb{R}_+^*)}}\leq\Vert I\Vert_{\mathrm{L}^{\infty}(\Omega\times\mathbb{R}_+^*)},
\]
and this quantity converges to $0$ as $\Vert(V,\tilde{P})\Vert_{\mathrm{L}^{\infty}(\Omega\times\mathbb{R}_+^*)}$ goes to $0$, since $I(x,t)=H(x)\left(1-\exp\left(-\beta_{VH}\int_0^tV_i(x,s)\mathrm{d}s\right)\right)$. Therefore, $IV_s=o(\Vert(V,\tilde{P})\Vert)$. Similarly, $VV_i=o(\Vert(V,\tilde{P})\Vert)$ and $\tilde{P}V_i=o(\Vert(V,\tilde{P})\Vert)$.

The population $V_i$ satisfies
\[
 \begin{array}{ll}
  \partial_tV_i & =\sigma_V\Delta V_i-(\alpha+d_V)V_i+\beta_{HV}IV_s-s_VVV_i-hPV_i\\
   & =\sigma_V\Delta V_i-(\alpha+d_V+h\dfrac{r_P}{s_P})V_i+\beta_{HV}IV_s-s_VVV_i-h\tilde{P}V_i\\
   & =\sigma_V\Delta V_i-(\alpha+d_V+h\dfrac{r_P}{s_P})V_i+o(\Vert(V,\tilde{P})\Vert)\\
   & =\sigma_V\Delta V_i-(\xi+mR)V_i+o(\Vert(V,\tilde{P})\Vert).
 \end{array}
\]

Denote
\[
 o\left(\int_0^{+\infty}\max(V,\tilde{P})\right)=\underset{\Vert(V,\tilde{P})\Vert_{\mathrm{L}^{\infty}(\Omega\times\mathbb{R}_+^*)}\to0}{o}\left(\int_0^{+\infty}\max(V,\tilde{P})\right)
\]
and
\[
 o\left(\int_{\Omega}\int_0^{+\infty}\max(V,\tilde{P})\right)=\underset{\Vert(V,\tilde{P})\Vert_{\mathrm{L}^{\infty}(\Omega\times\mathbb{R}_+^*)}\to0}{o}\left(\int_{\Omega}\int_0^{+\infty}\max(V,\tilde{P})\right).
\]

We have
\[
 \left\Vert \dfrac{\displaystyle\int_0^{+\infty}IV_s}{\displaystyle\int_0^{+\infty}\max(V,\tilde{P})}\right\Vert_{\mathrm{L}^{\infty}(\Omega)}\leq\Vert I\Vert_{\mathrm{L}^{\infty}(\Omega\times\mathbb{R}_+^*)},
\]
therefore, $\int_0^{+\infty}IV_s=o\left(\int_0^{+\infty}\max(V,\tilde{P})\right)$.

Similarly, $\int_0^{+\infty}VV_i=o\left(\int_0^{+\infty}\max(V,\tilde{P})\right)$ and $\int_0^{+\infty}\tilde{P}V_i=o\left(\int_0^{+\infty}\max(V,\tilde{P})\right)$. The same work can be done when integrating over $\Omega$.

Integrating in time between $0$ and $+\infty$, and because $V\underset{t\to+\infty}{\longrightarrow}0$,
\[
 -V_{i,0}=\sigma_V\Delta\displaystyle\int_0^{+\infty}V_i-(\xi+mR)\int_0^{+\infty}V_i+o\left(\int_0^{+\infty}\max(V,\tilde{P})\right).
\]

Integrating over $\Omega$,
\[
 \begin{array}{ll}
  -\displaystyle\int_{\Omega}V_{i,0}= & \sigma_V\displaystyle\int_{\Omega}\Delta\int_0^{+\infty}V_i-\xi\int_{\Omega}\int_0^{+\infty}V_i-m\int_{\Omega}\left(R\int_0^{+\infty}V_i\right)\\
  & +o\left(\displaystyle\int_{\Omega}\int_0^{+\infty}\max(V,\tilde{P})\right).
 \end{array}
\]

Because of the Neumann boundary conditions,
\[
 \int_{\Omega}\Delta\int_0^{+\infty}V_i=\left[\nabla\int_0^{+\infty}V_i\right]_{\partial\Omega}=0.
\]

As a consequence,
\[
 m\int_{\Omega}\left(R\int_0^{+\infty}V_i\right)=\displaystyle\int_{\Omega}V_{i,0}-\xi\int_{\Omega}\int_0^{+\infty}V_i+o\left(\int_{\Omega}\int_0^{+\infty}\max(V,\tilde{P})\right).
\]

Denote
\[
 o\left(\int_0^{+\infty}V_i\right)=\underset{\Vert(V,\tilde{P})\Vert_{\mathrm{L}^{\infty}(\Omega\times\mathbb{R}_+^*)}\to0}{o}\left(\int_0^{+\infty}V_i\right).
\]

On the other hand, using a Taylor expansion of the exponential function, the harvest verifies
\[
 \begin{array}{ll}
  \displaystyle\int_{\Omega}H-I_{\infty} & =\displaystyle\int_{\Omega}He^{-\beta_{VH}\int_0^{+\infty}V_i}\\
  & =\displaystyle\int_{\Omega}H\left(1-\beta_{VH}\int_0^{+\infty}V_i+o\left(\int_0^{+\infty}V_i\right)\right)\\
  & =\displaystyle\int_{\Omega}H-\beta_{VH}\displaystyle\int_{\Omega}\left(H\int_0^{+\infty}V_i\right)+o\left(\int_0^{+\infty}V_i\right),\\
 \end{array}
\]
where we used the boundedness of $\Omega$ and $H$ to get $\int_{\Omega}Ho\left(\int_0^{+\infty}V_i\right)=o\left(\int_0^{+\infty}V_i\right)$.

With $H=H^0(1-R)$,
{\small\[
 \begin{array}{lll}
  \displaystyle\int_{\Omega}H-I_{\infty} & = & |\Omega|H^0-H^0\displaystyle\int_{\Omega}R-\beta_{VH}H^0\displaystyle\int_{\Omega}\left(\int_0^{+\infty}V_i\right)\\
  && +\beta_{VH}H^0\displaystyle\int_{\Omega}\left(R\int_0^{+\infty}V_i\right)+o\left(\int_0^{+\infty}V_i\right)\\
  & = &|\Omega|H^0-H^0\displaystyle\int_{\Omega}R-\beta_{VH}H^0\displaystyle\int_{\Omega}\left(\int_0^{+\infty}V_i\right)\\
  && +\dfrac{\beta_{VH}H^0}{m}\left(\displaystyle\int_{\Omega}V_{i,0}-\xi\int_{\Omega}\int_0^{+\infty}V_i+o\left(\int_{\Omega}\int_0^{+\infty}\max(V,\tilde{P})\right)\right)\\
  && +o\left(\displaystyle\int_0^{+\infty}V_i\right)\\
  & = & |\Omega|H^0-H^0\displaystyle\int_{\Omega}R+\dfrac{\beta_{VH}H^0}{m}\displaystyle\int_{\Omega}V_{i,0}\\
  && -\beta_{VH}H^0(1+\dfrac{\xi}{m})\displaystyle\int_{\Omega}\left(\int_0^{+\infty}V_i\right)+o\left(\displaystyle\int_0^{+\infty}V_i\right).\\
 \end{array}
\]}

We have used the uniform Harnack inequality (see \cite{huska}, Theorem 2.5) to get
\[
 o\left(\displaystyle\int_{\Omega}\int_0^{+\infty}\max(V,\tilde{P})\right)=o\left(\displaystyle\int_0^{+\infty}V_i\right),
\]
which also gives
\[
 -\sigma_V\Delta\displaystyle\int_0^{+\infty}V_i+(\xi+mR)\int_0^{+\infty}V_i=V_{i,0}+o\left(\displaystyle\int_0^{+\infty}V_i\right).
\]

For a test function $\varphi\in\mathcal{C}^{2+\alpha}(\Omega)$ satisfying $\partial_n\varphi=0$, after integrating by part,
\[
 \displaystyle\int_{\Omega}\left(\left(-\sigma_V\Delta\varphi+(\xi+mR)\varphi\right)\int_0^{+\infty}V_i\right)=\displaystyle\int_{\Omega}\varphi V_{i,0}+o\left(\displaystyle\int_0^{+\infty}V_i\right).
\]

We now look for solutions of
\[
 -\sigma_V\Delta\varphi+(\xi+mR)\varphi=1
\]
supplemented with Neumann boundary conditions.

By theorem $6.31$ of \cite{gilbargtrudinger} there exists a unique solution in $\mathcal{C}^{2+\alpha}(\Omega)$. Denote $\varphi_R$ this solution. Injecting
\[
 \displaystyle\int_{\Omega}\left(\int_0^{+\infty}V_i\right)=\int_{\Omega}\varphi_R V_{i,0}+o\left(\displaystyle\int_0^{+\infty}V_i\right).
\]
in the equation of the harvest, one gets
\[
\begin{array}{ll}
 & \displaystyle\int_{\Omega}H-I_{\infty}+o\left(\displaystyle\int_0^{+\infty}V_i\right)\\
 = & |\Omega|H^0-H^0\displaystyle\int_{\Omega}R+\dfrac{\beta_{VH}H^0}{m}\displaystyle\int_{\Omega}V_{i,0}-\beta_{VH}H^0(1+\dfrac{\xi}{m})\int_{\Omega}\varphi_R V_{i,0}\\
 = & H^0\left(|\Omega|-\displaystyle\int_{\Omega}R\right)+\dfrac{\beta_{VH}H^0}{m}\left(\displaystyle\int_{\Omega}V_{i,0}-(\xi+m)\int_{\Omega}\varphi_R V_{i,0}\right).
\end{array}
\]

From now on, we denote
\[
 \eta_L(R,V_{i,0})=H^0\left(|\Omega|-\displaystyle\int_{\Omega}R\right)+\dfrac{\beta_{VH}H^0}{m}\left(\displaystyle\int_{\Omega}V_{i,0}-(\xi+m)\int_{\Omega}\varphi_R V_{i,0}\right),
\]
the linearized harvest.

\textbf{Important note:} The domain of applicability for the preceding computations leading to the linearized harvest, specifically the range of parameters denoted as $X$ where $\int_0^{+\infty}V_i$ is ``sufficiently small'', remains uncertain. For all the results presented in this paper, we assume that $X$ encompasses the entire set of admissible parameters where $\lambda_1(\mathcal{L}_{V_s})$ is positive. This assumption is likely flawed, as we conjectured that $\int_0^{+\infty}V_i$ diverges as $\lambda_1(\mathcal{L}_{V_s})$ approaches zero. We acknowledge that this assumption severs the connection between the linearized harvest and the actual harvest. Nevertheless, our results (with different boundary) should still hold within a closed, convex, and bounded set. Furthermore, we know that $X$ is not empty, as evidenced by the case where $V_0=0$ resulting in $\int_0^{+\infty}V_i=0$. Supposing that $X$ possesses non-empty interior within the space of all admissible parameters, then $X$ contains a closed, convex, and bounded region. We hope that the existence of this region renders the assumption more acceptable. Moreover, in numerical experiments involving a biologically consistent parameter set (details available in \cite{girardinmaucourt2023} for the sources and values of the chosen parameters), the harvest and linearized harvest are very close when $\lambda_1(\mathcal{L}_{V_s})$ is positive (even when it is small), as shown in Figure \ref{harvestfunofmu}. This suggests that $X$ covers a significant portion of the entire set of admissible parameters within the subset where $\lambda_1(\mathcal{L}_{V_s})$ is positive. In the following sections, we look for optimizers of the linearized harvest, and we hope that the numerical simulations provided suffice to convince the relevance of studying this quantity instead of the real harvest. As seen in Figures \ref{harvestfunofmubigHV} and \ref{harvestfunofmubigVH}, when the transmission rates are very large, the two harvests become further apart, which is consistent with the fact that $\int_0^{+\infty}V_{i,0}$ becomes larger as $\beta_{HV},\beta_{VH}\to+\infty$.

\begin{center}
 \begin{figure}
  \includegraphics[scale=0.8]{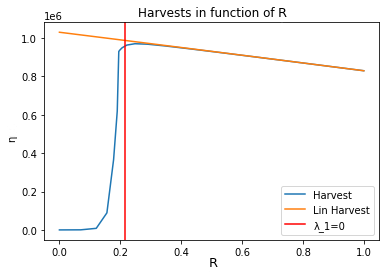}
  \caption{Harvest and linearized harvest in function of a constant $R$ (with biologically consistent parameters, see \cite{girardinmaucourt2023}). Explicit formula for $\eta_L$ can be found in Remark 1.}
  \label{harvestfunofmu}
 \end{figure}
\end{center}

\begin{center}
 \begin{figure}
  \includegraphics[scale=0.8]{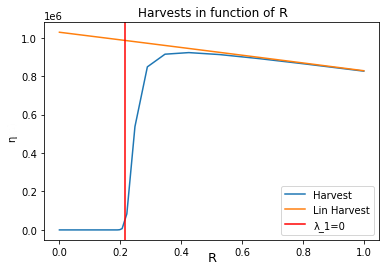}
  \caption{Harvest and linearized harvest in function of a constant $R$ (with $\beta_{HV}$ ten times bigger).}
  \label{harvestfunofmubigHV}
 \end{figure}
\end{center}

\begin{center}
 \begin{figure}
  \includegraphics[scale=0.8]{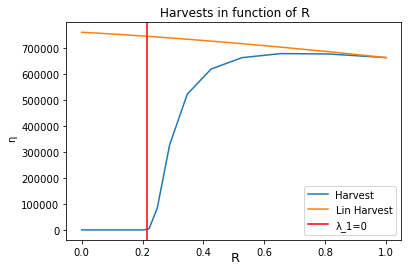}
  \caption{Harvest and linearized harvest in function of a constant $R$ (with $\beta_{VH}$ ten times bigger).}
  \label{harvestfunofmubigVH}
 \end{figure}
\end{center}

\textbf{Remark 1:} In figure \ref{harvestfunofmu}, we compute the harvest $\eta$ and linearized harvest $\eta_L$ in the case where $R$ and $V_{i,0}$ are constant in space. Here, $\eta_L$ is explicit since $\varphi_R$ can be computed: $\varphi_R=1/(\xi+mR)$ yields
\[
 \begin{array}{ll}
  \eta_L(R,V_{i,0}) & =H^0\left(|\Omega|-\displaystyle\int_{\Omega}R\right)+\dfrac{\beta_{VH}H^0}{m}\left(1-\dfrac{\xi+m}{\xi+mR}\right)\displaystyle\int_{\Omega}V_{i,0}\\
  & =H^0|\Omega|\left(1-R\right)+\dfrac{\beta_{VH}H^0}{m}\left(\dfrac{mR-m}{\xi+mR}\right)|\Omega|V_{i,0}\\
  & =H^0|\Omega|(1-R)\left(1-\dfrac{\beta_{VH}V_{i,0}}{\xi+mR}\right).
 \end{array}
\]
Although it looks like a linear function of $R$, it is not. With our parameter values (see \cite{girardinmaucourt2023} for some details), $\eta_L(R)\approx1058823(1-0.2R)\left(1-\dfrac{0,0448}{1.58+0,468R}\right)$. We have taken an initial condition $V_{i,0}$ of one percent of the carrying capacity of $V$.

\textbf{Remark 2:} One can observe that the harvest is positive for some negative values of $\lambda_1(\mathcal{L}_{V_s})$. This is, to our understanding, due to numerical imprecisions: when $\lambda_1(\mathcal{L}_{V_s})$, while still negative, becomes very close to $0$, $V_i$ diverges to $+\infty$ much slower (at a rate that is roughly $e^{-\lambda_1(\mathcal{L}_{V_s})t}$), therefore $e^{-\int_0^{+\infty}V_i}$ (that we approximate by computing $e^{-\int_0^T V_i}$ for a large $T$) becomes at a greater distance from $0$. This phenomenon is less visible when we increase $T$, as can be observed in Figure \ref{harvestfunofmubigT}.

\begin{center}
 \begin{figure}
  \includegraphics[scale=0.8]{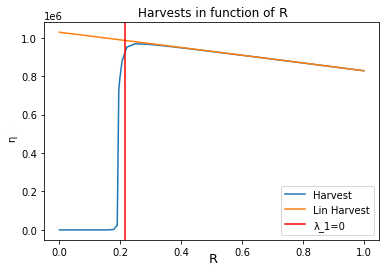}
  \caption{Harvest and linearized harvest in function of a constant $R$ (with $T$ three times bigger).}
  \label{harvestfunofmubigT}
 \end{figure}
\end{center}

We denote $\mathcal{C}_b^{2+\alpha}(\Omega)$ the space of bounded $\mathcal{C}^{2+\alpha}$ functions defined on $\Omega$. In the proofs of Lemma \ref{concavitylinharvest} and Theorem \ref{optconst}, we use differential calculus in Banach spaces (in $\mathrm{L}^{\infty}(\Omega)$ and in $\mathcal{C}_b^{2+\alpha}(\Omega)$). As the computation of the partial derivative of $\varphi_R$ or $\eta_L$ with respect to $R$ is usual, we freely use it without delving into its details.

In the context of the optimization problem, the first crucial lemma is the concavity of the application $R\mapsto\eta_L(R,V_{i,0})$:

\begin{lemma}\label{concavitylinharvest}
 Fix $V_{i,0}\in\mathcal{C}^{2+\alpha}(\Omega)$. Then, the application
 \[
  \left\lbrace \begin{array}{lll}
   \mathcal{C}_b^{2+\alpha}(\Omega) & \to & \mathbb{R}\\
   R & \mapsto & \eta_L(R,V_{i,0})
  \end{array}\right.
 \]
 is concave and coercive : $\eta_L(R,V_{i,0})\underset{\Vert R\Vert_{\mathcal{C}_b^{2+\alpha}(\Omega)}\to+\infty}{\longrightarrow}-\infty$.
\end{lemma}

\begin{proof}
 By Lemma 1.4 of Chapter 3 of \cite{pao} applied on
 \[
  -\sigma_V\Delta\varphi_R+(\xi+mR)\varphi_R=1,
 \]
 the function $\varphi_R$ is positive in $\Omega$.
 
 By the same lemma, we obtain consequently the negativity of $\partial_R\varphi_R$, solving
 \[
  -\sigma_V\Delta\partial_R\varphi_R+(\xi+mR)\partial_R\varphi_R=-m\varphi_R,
 \]
 and the positivity of $\partial_R^2\varphi_R$ in $\Omega$. The application $R\mapsto\eta_L(R,V_{i,0})$ is hereby concave.
 
 The linearized harvest has the following form:
 \[
  \eta_L(R,V_{i,0})=C_1+C_2\int_{\Omega}V_{i,0}-C_3\int_{\Omega}R-C_4\int_{\Omega}\varphi_RV_{i,0},
 \]
 with $C_1,C_2,C_3,C_4>0$. Hence, to prove its coercivity, it suffices to show the boundedness of $\int_{\Omega}\varphi_RV_{i,0}$ when $\Vert R\Vert_{\mathcal{C}_b^{2+\alpha}(\Omega)}\to+\infty$. By nonnegativity of $m$ and $R$, $1/\xi$ is a supersolution of the Neumann system solved by $\varphi_R$, hence $\varphi_R\leq 1/\xi$ and the coercivity of $\eta_L$ follows.

\end{proof}

\textbf{Comment:} The concavity of $R\mapsto\eta_L(R,V_{i,0})$ gives interesting insights on the optimal refuge, $R_{\textnormal{opt}}^L:=\underset{R\in\mathcal{C}^{2+\alpha}(\Omega,[0,1])}{\operatorname{argmax}}\;\eta_L(R,V_{i,0})$. For instance, there exists $\beta_{VH}$ small enough such that
\[
 \partial_R\eta_L=-H^0\int_{\Omega}\cdot-\beta_{VH}H^0(1+\dfrac{\xi}{m})\int_{\Omega}\partial_R\varphi_RV_{i,0}\cdot
\]
is negative, and in this case $R_{\textnormal{opt}}^L\equiv 0$. Conversely, there exists $\beta_{VH}$ large enough such that $\partial_R\eta_L$ is positive (by negativity of $\partial_R\varphi_R$), and in this case $R_{\textnormal{opt}}^L\not\equiv 0$. The same comment can be made with $\Vert V_{i,0}\Vert$ instead of $\beta_{VH}$.

A natural question can arise on the convergence of the linearized harvest when the frequency of the refuge goes to infinity. We answer the question in the following proposition.

\begin{proposition}
 In the context of the homogenization problem (see Section 3 for the details), $\eta_L(R_n,V_{i,0})\underset{n\to+\infty}{\longrightarrow}\eta_L(\mathcal{R},V_{i,0})$.
\end{proposition}

\begin{proof}
 The proof lies in the fact that $\int_{\Omega}\varphi_{R_n} V_{i,0}\underset{n\to+\infty}{\longrightarrow}\int_{\Omega}\varphi_{\mathcal{R}} V_{i,0}$.
 
 Integrating
 \[
  -\sigma_V\Delta\varphi_{R_n}+(\xi+mR_n)\varphi_{R_n}=1,
 \]
 one gets
 \[
 \begin{array}{ll}
  \xi\int_{\Omega}\varphi_{R_n} & =|\Omega|-m\int_{\Omega}R_n\varphi_{R_n}\\
  & \leq |\Omega|.
 \end{array}
 \]
 
 Multiplying the equation on $\varphi_{R_n}$ by $\varphi_{R_n}$ and integrating over $\Omega$, one gets
 \[
  \sigma_V\int_{\Omega}|\nabla\varphi_{R_n}|^2+\int_{\Omega}(\xi+mR_n)\varphi_{R_n}^2=\int_{\Omega}\varphi_{R_n}\leq |\Omega|/\xi,
 \]
 hence the uniform boundedness of $\Vert\varphi_{R_n}\Vert_{\mathrm{H}^1(\Omega)}$. A direct application of the Rellich–Kondrachov theorem yields, up to extraction, the strong $\mathrm{L}^2$-convergence of $\varphi_{R_n}$ to $\overline{\varphi}$.
 
 Let $\psi\in\mathcal{C}^{2+\alpha}(\Omega)$. Multiplying the equation on $\varphi_{R_n}$ by $\psi$ and integrating over $\Omega$, one gets
 \[
  -\sigma_V\int_{\Omega}\psi\Delta\varphi_{R_n}+\int_{\Omega}(\xi+mR_n)\psi\varphi_{R_n}=\int_{\Omega}\psi
 \]
 and two integration by parts followed by passing to the limit $n\to+\infty$ yields
 \[
  -\sigma_V\int_{\Omega}\Delta\psi\overline{\varphi}+\int_{\Omega}(\xi+m\mathcal{R})\psi\overline{\varphi}=\int_{\Omega}\psi.
 \]
 
 By uniqueness of the weak solution of this equation, $\overline{\varphi}=\varphi_{\mathcal{R}}$.
\end{proof}

\subsection{Schwarz rearrangement}

In this section, we assume that $\Omega\subset\mathbb{R}^d$ is symmetric. In order to alleviate the proofs, we will further assume that $d=1$, hence $\Omega=[-L,L]$, $L>0$. However, it is not, as far as we know, a necessary assumption. For an arbitrary $f\in\mathrm{L}^{\infty}(\Omega)$, we will denote by $f_*$ the decreasing Schwarz rearrangement of $f$, which is a symmetric, radially nonincreasing function that preserves the measures of the level sets of $f$ (see \cite{talenti1994} and \cite{talenti2016} for a thorough definition of the Schwarz rearrangement). With our assumptions on $\Omega$, $f_*$ also belongs in $\mathrm{L}^{\infty}(\Omega)$.

We call the increasing Schwarz rearrangement of $f$, and denote $f^*$, the application 
\[
 f^*:\left\lbrace \begin{array}{cll}
      \Omega & \to & \mathbb{R}\\
      x & \mapsto & f_*(L-|x|).
     \end{array}\right.
\]

Our main theorem states that, if $V_{i,0}$ is its own decreasing Schwarz rearrangement and is decreasing, then the optimal refuge for the linearized harvest is not constant, and there exists a refuge that is its own decreasing Schwarz rearrangement which yields a better linearized harvest than the constant refuge. The same proof also works if $V_{i,0}$ is its own increasing Schwarz rearrangement and is increasing, and the result would be that the optimal refuge for the linearized harvest is not constant, and there exists a refuge that is its own increasing Schwarz rearrangement which yields a better linearized harvest than the constant refuge.

\begin{theorem}[Theorem~\ref{threarrange}]\label{thschwarz}
 Denote, for $c>0$, 
 
 $\mathcal{M}_c(\Omega)=\{f\in\mathcal{C}^{2+\alpha}(\overline{\Omega},[0,1]),\, \Vert f\Vert_{\mathrm{L}^1}=c\}$.
 
 Fix $V_{i,0}\in\mathcal{C}^{2+\alpha}(\Omega)$ and $\mathcal{R}>0$. Let us assume that $V_{i,0}=(V_{i,0})_*$ and is radially decreasing. Then,
 \[
  \eta_L(\mathcal{R},V_{i,0})\neq\underset{R\in\mathcal{M}_{|\Omega|\mathcal{R}}(\Omega)}{\max}\eta_L(R,V_{i,0}).
 \]
 
 Moreover, there exists $R\in\mathcal{M}_{|\Omega|\mathcal{R}}(\Omega)$ such that $R=R_*$ and $\eta_L(R,V_{i,0})>\eta_L(\mathcal{R},V_{i,0})$.
\end{theorem}

\begin{proof}
 Recall that
 \[
  \eta_L(R,V_{i,0})=H^0\left(|\Omega|-\displaystyle\int_{\Omega}R\right)+\dfrac{\beta_{VH}H^0}{m}\left(\displaystyle\int_{\Omega}V_{i,0}-(\xi+m)\int_{\Omega}\varphi_R V_{i,0}\right),
 \]
 hence the quantity we aim to minimize is
 \[
  \int_{\Omega}\varphi_R V_{i,0}
 \]
 with
 \[
  -\sigma_V\varphi_R''+(\xi+mR)\varphi_R=1.
 \]
 
 Let $\tilde{R}\in\mathcal{M}_0(\Omega)$. Denote $\psi:\varepsilon\mapsto\varphi_{\mathcal{R}+\varepsilon\tilde{R}}\in\mathcal{C}^2([0,+\infty),\mathcal{M}_{|\Omega|\mathcal{R}}(\Omega))$ and $u=\psi'(0)$. Then, for all $\varepsilon>0$, $\varphi_{\mathcal{R}+\varepsilon\tilde{R}}=\varphi_{\mathcal{R}}+\varepsilon u+o(\varepsilon)$. With $\varphi_{\mathcal{R}}=1/(\xi+m\mathcal{R})$, injecting this Taylor expansion in the equation of $\varphi_{\mathcal{R}+\varepsilon\tilde{R}}$, one gets, at first order,
 \[
  -\sigma_V(\varphi_{\mathcal{R}}+\varepsilon u)''+(\xi+m(\mathcal{R}+\varepsilon\tilde{R}))(\varphi_{\mathcal{R}}+\varepsilon u)=1,
 \]
 hence, dividing by $\varepsilon$ and letting it go to $0$,
 \[
  -\sigma_Vu''+(\xi+m\mathcal{R})u=-m\tilde{R}/(\xi+m\mathcal{R}).
 \]
 
 By Schwarz theorem, $u_{|\partial\Omega}=\dfrac{\partial}{\partial\varepsilon}((\varphi_{\mathcal{R+\varepsilon\tilde{R}}})_{|\partial\Omega})(0)$. Solving this ordinary differential equation with Neumann boundary conditions on $\Omega=[-L,L]$ with $\tilde{R}(x)=\cos\left(\dfrac{\pi}{L}x\right)$, one gets
 \[
  u(x)=\dfrac{-m}{(\xi+m\mathcal{R})(\xi+m\mathcal{R}+\sigma_V(\pi/L)^2)}\cos\left(\dfrac{\pi}{L}x\right).
 \]
 
 By symmetry of $u$ and $V_{i,0}$,
 \[
  \int_{\Omega}uV_{i,0}=2\int_0^{L/2}uV_{i,0}+2\int_{L/2}^LuV_{i,0}.
 \]
 Because $u(L/2-x)=-u(L/2+x)$ for all $x\in[0,L/2]$, after a change of variable, one gets
 \[
  \int_0^{L/2}uV_{i,0}=-\int_{L/2}^Lu(x)V_{i,0}(L-x)\mathrm{d}x
 \]
 and therefore, because $V_{i,0}$ is decreasing on $[0,L]$,
 \[
  \int_{\Omega}uV_{i,0}=2\int_{L/2}^Lu(x)(V_{i,0}(x)-V_{i,0}(L-x))\mathrm{d}x<0.
 \]
 
 Moreover,
 \[
  \int_{\Omega}\varphi_{\mathcal{R}+\varepsilon\tilde{R}}V_{i,0}=\int_{\Omega}\varphi_{\mathcal{R}}V_{i,0}+\varepsilon\int_{\Omega}uV_{i,0}+o(\varepsilon),
 \]
 hence for a small enough $\varepsilon$, $\int_{\Omega}\varphi_{\mathcal{R}+\varepsilon\tilde{R}}V_{i,0}<\int_{\Omega}\varphi_{\mathcal{R}}V_{i,0}$, i.e $\eta_L(\mathcal{R}+\varepsilon\tilde{R},V_{i,0})>\eta_L(\mathcal{R},V_{i,0})$. One can easily verify that $\mathcal{R}+\varepsilon\tilde{R}\in\mathcal{M}_{|\Omega|\mathcal{R}}(\Omega)$ and $(\mathcal{R}+\varepsilon\tilde{R})_*=\mathcal{R}+\varepsilon\tilde{R}$.

\end{proof}

The next corollary is similar: if we fix the refuge $R$, and if $R$ is its own Schwarz rearrangement, then the optimal $V_{i,0}$ for the linearized harvest is its own Schwarz rearrangement.

\begin{corollary}\label{corschwarz}
 Fix $R\in\mathcal{C}^{2+\alpha}(\Omega)$.
 
 Denote, for $c_1,c_2>0$, $\mathcal{M}_{c_1,c2}(\Omega)=\{f\in \mathrm{L}^{\infty}(\Omega),\, \Vert f\Vert_{\mathrm{L}^1}=c_1,\,\Vert f\Vert_{\mathrm{L}^{\infty}}\leq c_2\}$. The quantities
 \[
  m_V:=\underset{V_{i,0}\in\mathcal{M}_{c_1,c_2}(\Omega)}{\min}\eta_L(R,V_{i,0})
 \]
 and
 \[
  M_V:=\underset{V_{i,0}\in\mathcal{M}_{c_1,c_2}(\Omega)}{\max}\eta_L(R,V_{i,0})
 \]
 exist. Moreover, if $R=R_*$ and is not constant, then
 \[
  \{V_{i,0}\in\mathcal{M}_{c_1,c2}(\Omega),\,\eta_L(R,V_{i,0})=m_V\}\subset\{V_{i,0}\in\mathcal{M}_{c_1,c2}(\Omega),\,V_{i,0}=(V_{i,0})^*\}
 \]
 and
 \[
  \{V_{i,0}\in\mathcal{M}_{c_1,c2}(\Omega),\,\eta_L(R,V_{i,0})=M_V\}\subset\{V_{i,0}\in\mathcal{M}_{c_1,c2}(\Omega),\,V_{i,0}=(V_{i,0})_*\}.
 \]
 
\end{corollary}

\begin{proof}
 With this set of constraints, one optimizes $\eta_L$ by optimizing the quantity\linebreak $\int_{\Omega}\varphi_R V_{i,0}$. Let $(\underline{V}_n)$ be a minimizing sequence, and $(\overline{V}_n)$ be a maximizing sequence of $V\mapsto\int_{\Omega}\varphi_R V$. The set $\mathcal{M}_{c_1,c_2}(\Omega)$ is a closed, convex and bounded subset of $\mathrm{L}^{\infty}(\Omega)$, therefore it is weakly compact. Hence, up to an extraction, the sequences converge to $\underline{V}_{\infty},\overline{V}_{\infty}\in\mathcal{M}_{c_1,c_2}$ minimizing and maximizing $\eta_L$.
 
 Let us suppose that $R=R_*$. By positivity of the principal eigenvalue of $-\sigma_V\Delta+\xi+mR$, the energy
 \[
 \mathcal{E}:\left\lbrace \begin{array}{rcl}
              \mathcal{C}^{2+\alpha}(\Omega,\mathbb{R}_+) & \to & \mathbb{R}\\
              u & \mapsto & \dfrac{1}{2}\left(\displaystyle\int_{\Omega}\sigma_V|\nabla u|^2+(\xi+mR)u^2\right)-\displaystyle\int_{\Omega}u
             \end{array}\right.
 \]
 is coercive, therefore $\varphi_R$ is the unique minimizer of $\mathcal{E}$.
 
 Moreover, by the general rearrangement inequality (see e.g Theorem 3.4 of \cite{brascamp1974}), for all $u\in\mathcal{C}^{2+\alpha}(\Omega,\mathbb{R}_+)$ and for a large constant $C>0$,
 \[
  \int_{\Omega}(C-R)^*(u^*)^2\geq\int_{\Omega}(C-R)u^2.
 \]
 
 By the fifth property of part $3$ of \cite{talenti2016} (page 10),
 \[
  \int_{\Omega}(u^*)^2=\int_{\Omega}u^2,
 \]
 which yields
 \[
  \int_{\Omega}R(u^*)^2\leq\int_{\Omega}Ru^2.
 \]
 
 By the Pólya–Szegő inequality (see \cite{talenti1994}), for all $u\in\mathcal{C}^{2+\alpha}(\Omega,\mathbb{R}_+)$,
 \[
  \int_{\Omega}|\nabla u^*|^2\leq\int_{\Omega}|\nabla u|^2.
 \]
 
 The previous inequalities yield, for all $u\in\mathcal{C}^{2+\alpha}(\Omega,\mathbb{R}_+)$,
 \[
  \mathcal{E}(u^*)\leq\mathcal{E}(u),
 \]
 hence, by uniqueness of the minimizer of $\mathcal{E}$, $\varphi_R=\varphi_R^*$.
 
 By Talenti inequalities \cite{talenti1994}, for all $V_{i,0}\in\mathcal{M}_{c_1,c_2}(\Omega)$,
 \[
  \int_{\Omega}\varphi_RV_{i,0}\leq\int_{\Omega}\varphi_R(V_{i,0})^*
 \]
 and
 \[
  \int_{\Omega}\varphi_RV_{i,0}\geq\int_{\Omega}\varphi_R(V_{i,0})_*.
 \]
 
 Let $V_{i,0}\in\mathcal{M}_{c_1,c_2}(\Omega)$ such that $\eta_L(R,V_{i,0})=m_V.$ Then $V_{i,0}$ maximizes the quantity $\int_{\Omega}\varphi_R V_{i,0}$. By the first inequality above,
 \[
  \int_{\Omega}\varphi_RV_{i,0}=\int_{\Omega}\varphi_R(V_{i,0})^*.
 \]
 
 By our assumption on $R$, $\varphi_R$ is increasing. By \cite{hajaiej2004}, $V_{i,0}=(V_{i,0})^*$. The case $\eta_L(R,V_{i,0})=M_V$ is treated identically, with $\varphi_R$ being decreasing.
\end{proof}

\subsection{Explicit optimizers in the homogeneous case}

In this section, we assume that, for all $x\in\Omega$, $V_{i,0}(x)$ follows a given probability distribution $\mathbb{P}_x$ on $\mathbb{R}_+$, such that its mean $x\mapsto\mathbb{E}[V_{i,0}(x)]$ is known and homogeneous in space. Since this section is not intended to focus on probability concepts, we intentionally refrain from delving into a rigorous definition of this probability distribution. The biological rationale for this approximation is as follows: the primary provenance of aphids in the field is wind dispersal. Aphids transported by the wind cover distances significantly greater than the size of an average beet field, making it reasonable to employ a uniform mean for initially infected vectors.

The mean of the linearized harvest is
{\small\[
 \mathbb{E}[\eta_L(R,V_{i,0})]=H^0\left(|\Omega|-\displaystyle\int_{\Omega}R\right)+\dfrac{\beta_{VH}H^0}{m}\left(\displaystyle\int_{\Omega}\mathbb{E}[V_{i,0}]-(\xi+m)\int_{\Omega}\varphi_R \mathbb{E}[V_{i,0}]\right).
\]}

We prove in the following that the constant refuge $\mathcal{R}$ is optimal, for the mean of the linearized harvest $\mathbb{E}[\eta_L]$, in the space of bounded $\mathcal{C}^{2+\alpha}$ functions defined on $\Omega$ we denote $\mathcal{C}^{2+\alpha}_b(\Omega,\mathbb{R})$. Note that rather than assuming a constant mean for $V_{i,0}$ in space, an alternative approach is to assume $V_{i,0}$ itself to be constant. This would yield a similar result, asserting that a constant refuge maximizes the linearized harvest. The choice of using the mean is purely biological; it renders the assumption about the initial condition more realistic, since having $V_{i,0}$ constant in space appears highly unlikely in real-life scenarios.

\begin{theorem}[Theorem~\ref{thconstant}]\label{optconst}
 Suppose that $x\mapsto\mathbb{E}[V_{i,0}]$ is constant. Let $\mathcal{R}=\left(\dfrac{\sqrt{\beta_{VH}(\xi+m)\mathbb{E}[V_{i,0}]}-\xi}{m}\right)^+$. Then
 \[
  \mathbb{E}[\eta_L(\mathcal{R},V_{i,0})]=\underset{R\in\mathcal{C}^{2+\alpha}_b(\Omega,\mathbb{R})}{\max}\mathbb{E}[\eta_L(R,V_{i,0})]
 \]
\end{theorem}

\begin{proof}
First of all, by Lemma \ref{concavitylinharvest}, the application $R\mapsto\eta_L(R,V_{i,0})$ is concave and coercive, which immediately yields the existence and uniqueness of the maximizer of $\eta_L$, which nullifies $\partial_R\mathbb{E}[\eta_L(R,V_{i,0})]=\partial_R\eta_L(R,\mathbb{E}[V_{i,0}])$.

Let $R\in\mathcal{C}^{2+\alpha}_b(\Omega,\mathbb{R})$ so that $\partial_R\mathbb{E}[\eta_L(R,V_{i,0})]=0$, and let $\zeta\in\mathcal{C}^{2+\alpha}_b(\Omega,\mathbb{R})$. We have
\[
 \begin{array}{ll}
  \partial_R\mathbb{E}[\eta_L(R,V_{i,0})]\cdot\zeta & =-H^0\int_{\Omega}\zeta-\beta_{VH}H^0(1+\dfrac{\xi}{m})\mathbb{E}[V_{i,0}]\int_{\Omega}\partial_R\varphi_R\zeta\\
  & =0,
 \end{array}
\]
hence
\[
 \int_{\Omega}\partial_R\varphi_R\zeta=-\dfrac{m}{\beta_{VH}(\xi+m)\mathbb{E}[V_{i,0}]}\int_{\Omega}\zeta
\]
 
Denote $C=-\dfrac{m}{\beta_{VH}(\xi+m)\mathbb{E}[V_{i,0}]}$. There exists a sequence of step functions $(\psi_k)$ such that $\Vert\psi_k-\partial_R\varphi_R\zeta\Vert_{\mathrm{L}^2}\to 0$. Let $h$ be a function in $\mathrm{L}^2(\Omega)$.
\[
 \begin{array}{ll}
  |\int_{\Omega}(\psi_k-C)h| & =|\int_{\Omega}(\psi_k-\partial_R\varphi_R\zeta+\partial_R\varphi_R\zeta-C)h|\\
  & =|\int_{\Omega}(\psi_k-\partial_R\varphi_R\zeta)h|\\
  & \leq\Vert\psi_k-\partial_R\varphi_R\zeta\Vert_{\mathrm{L}^2}\sqrt{\int_{\Omega}h^2}\\
  & \underset{k\to+\infty}{\longrightarrow}0.
 \end{array}
\]
 
We have shown the weak convergence in $\mathrm{L}^2$ of $(\psi_k)$ to $C$. By uniqueness of the limit, $\partial_R\varphi_R\zeta=C$. By differentiating the equation on $\varphi_R$ with respect to $R$, one gets
\[
 -\sigma_V\Delta\partial_R\varphi_R\zeta+(\xi+mR)\partial_R\varphi_R\zeta=-m\varphi_R\zeta,
\]
therefore
\[
 \varphi_R=\dfrac{\xi+mR}{\beta_{VH}(\xi+m)\mathbb{E}[V_{i,0}]}.
\]
 
Plugging the expression of $\varphi_R$ into its equation, it yields
\[
 \left\lbrace \begin{array}{ll}
  -\sigma_Vm\Delta R+(\xi+mR)^2=\beta_{VH}(\xi+m)\mathbb{E}[V_{i,0}]\\
  \partial_n R=0.
 \end{array}\right.
\]

By uniqueness of the solution of this elliptic problem with Neumann boundary conditions (see e.g Section 3.3 of \cite{pao}), we deduce
\[
 R=\dfrac{\sqrt{\beta_{VH}(\xi+m)\mathbb{E}[V_{i,0}]}-\xi}{m}.
\]

Reciprocally, if $R=\dfrac{\sqrt{\beta_{VH}(\xi+m)\mathbb{E}[V_{i,0}]}-\xi}{m}$, following the proof backwards, it is clear that $\partial_R\mathbb{E}[\eta_L(R,V_{i,0})]=0$. If this quantity is negative, i.e if $\beta_{VH}(\xi+m)\mathbb{E}[V_{i,0}]<\xi^2$, then a straightforward computation yields $\partial_R\mathbb{E}[\eta_L(0,V_{i,0})]\cdot\zeta=-H^0\int_{\Omega}\zeta+\beta_{VH}H^0(1+\dfrac{\xi}{m})\mathbb{E}[V_{i,0}]\dfrac{m}{\xi^2}\int_{\Omega}\zeta<0$. Therefore, in this case, the optimal refuge is $0$.
\end{proof}

\textbf{Remark:} We just proved that, if $V_{i,0}$ is constant, then $\partial_R\eta_L=0\iff R=\mathcal{R}$. With a very similar proof, we can obtain that, if $R$ is constant, then $\partial_R\eta_L=0\iff V_{i,0}=\dfrac{(\xi+mR)^2}{\beta_{VH}m(1+\xi/m)}$. We deduce that, if $R$ is constant and $V_{i,0}$ is not constant, then $\eta_L(R,V_{i,0})$ is not optimal.

\bibliographystyle{plain}
\bibliography{biblioBaptiste}

\end{document}